\documentclass[12pt]{article}
\usepackage{graphicx} 
\usepackage[top=2cm,bottom=2cm,left=2cm,right=2cm]{geometry}

\usepackage{palatino}
\usepackage{latexsym}

\usepackage[mathscr]{eucal}
\usepackage{amsmath}
\usepackage{amsthm}
\usepackage{amsfonts}
\usepackage{amssymb}
\usepackage{amscd}
\usepackage{color}
\usepackage{graphicx}
\usepackage{graphics}
\usepackage{pifont}
\usepackage{subfigure}
\usepackage[makeroom]{cancel}
\usepackage[normalem]{ulem}
\usepackage[dvipsnames]{xcolor}
\usepackage{tikz}
\usepackage{cite}
\usepackage[hidelinks,pdfusetitle]{hyperref}

\theoremstyle{plain}
\newtheorem{theorem}{Theorem }
\newtheorem{remark}[theorem]{Remark}
\newtheorem{lemma}[theorem]{Lemma}
\newtheorem{proposition}[theorem]{Proposition}
\newtheorem{definition}[theorem]{Definition}

\newcommand{\R}{{\mathbb R}}
\newcommand{\N}{{\mathbb N}}
\newcommand{\T}{{\mathbb T}}
\newcommand{\C}{{\mathbb C}}
\newcommand{\la}{{\langle}}
\newcommand{\ra}{{\rangle}}

\title{Examples of small-time controllable Schrödinger equations\\
}

\begin{document}
\author{Karine Beauchard\footnote{Univ Rennes, CNRS, IRMAR - UMR 6625, F-35000 Rennes, France (karine.beauchard@ens-rennes.fr)},\quad Eugenio Pozzoli\footnote{Univ Rennes, CNRS, IRMAR - UMR 6625, F-35000 Rennes, France (eugenio.pozzoli@univ-rennes.fr)}}
\maketitle

\abstract{
A variety of physically relevant bilinear Schr\"odinger equations are known to be approximately controllable in large times. There are however examples which are approximately controllable in large times, but not in small times. This obstruction happens e.g. in the presence of (sub)quadratic potentials, because Gaussian states are preserved, at least for small times.

In this work, we provide the first examples of small-time approximately controllable bilinear Schrödinger equations. In particular, we show that a control on the frequency of a quadratic potential permits to construct approximate solutions that evolve arbitrarily fast along space-dilations. Once we have access to space-dilations, we can exploit them to generate time-contractions. In this way, we build on previous results of large-time control, to obtain control in small times.

}






\section{Introduction}
\subsection{The model}

Let $M$ be a Riemannian manifold without boundary. We consider the initial value problem for Schrödinger equations of the form
\begin{equation}\label{eq:schro}
\begin{cases}
i\partial_t\psi(t,x)=\left(-\Delta+V(x)+\sum_{j=1}^m u_j(t)W_j(x)\right)\psi(t,x), & (t,x) \in (0,T) \times M, \\
\psi(0,\cdot)=\psi_0.
\end{cases}
\end{equation}
The linear unbounded operator $\Delta$ is the Laplace-Beltrami operator of $M$, and is defined on the domain $H^2(M,\C)$. The time-dependent potential $V(x)+\sum_{j=1}^m u_j(t)W_j(x)$ is possibly unbounded, and defined on a suitable domain. The time-independent part $-\Delta+V$ is usually referred to as \emph{the drift}.

System \eqref{eq:schro} describes the dynamics of a quantum particle on the manifold $M$, with free (kinetic plus potential) energy $-\Delta+V$, in interaction with additional external fields with potentials $W_j$ that can be switched on and off. It is used to model a variety of physical situations, such as atoms in optical cavities, and molecular dynamics.

When well defined, the solution of \eqref{eq:schro} at time $t$, associated with a time-dependent function $u(t)=(u_1(t),\dots,u_m(t))$, and an initial state $\psi_0$, is denoted by $\psi(t;u,\psi_0)$ and lives in the unitary sphere $\mathcal{S}$ of $L^2(M)$
\begin{equation} \label{def:S}
\mathcal{S}:=\{\psi\in L^2(M,\C)\, ;\, \|\psi\|_{L^2(M)}=1\}.
\end{equation}

We are interested in the controllability properties of \eqref{eq:schro}. We think of the function $u$ as a \emph{control} that can be chosen and implemented in order to change the \emph{state} $\psi$ of the system. More precisely, given an initial and a final state $\psi_0,\psi_1$, we would like to find a control which steers the system from $\psi_0$ to $\psi_1$. We are in particular interested in the family of states that can be approximately reached from $\psi_0$ in arbitrarily small times (i.e., in time approximately zero).

\begin{definition}[Small-time controllability]
Let $(\mathcal{H},\|.\|_{\mathcal{H}})$ be a normed space, subset of $L^2(M,\C)$. 

For $\psi_0, \psi_1 \in \mathcal{H} \cap \mathcal{S}$, we say that $\psi_1$ is small-time 
$\mathcal{H}$-approximately (resp. exactly)
reachable from $\psi_0$ 
if for any $\varepsilon >0$, there exist 
a time $T \in[0,\varepsilon]$,
and a 
control $u:[0, T ]\to \R^m$ such that the Cauchy problem \eqref{eq:schro} 
has a unique solution $\psi \in C^0([0,T],\mathcal{H})$ and
$\|\psi( T;u,\psi_0)- \psi_1\|_{\mathcal{H}}<\varepsilon$
(resp. $\psi(T;u,\psi_0)= \psi_1$).
Then $\text{Adh}_{\mathcal{H}} \text{Reach}_{\text{st}}(\psi_0)$ 
(resp. $\text{Reach}_{\text{st}}(\psi_0)$)
denotes the set of small-time 
$\mathcal{H}$-approximately (resp. exactly) reachable states from $\psi_0$.

We say that \eqref{eq:schro} is small-time $\mathcal{H}$-approximately controllable if, for every
$\psi_0\in \mathcal{H} \cap \mathcal{S}$, 
$\mathcal{H} \cap \mathcal{S}=
\text{Adh}_{\mathcal{H}} \text{Reach}_{\text{st}}(\psi_0)$.
\end{definition}


Many equations of the form (\ref{eq:schro}) are known to be approximately controllable in large time, in various functional frameworks $\mathcal{H}$ (see Section \ref{subsec:Biblio}).

The control of Schrödinger equations \textbf{in small time} is an open challenge: in particular, there exist examples of equations as \eqref{eq:schro} which are $\mathcal{H}$-approximately controllable in large times, but not in small-time (see Section \ref{sec:STAC}). In this paper, we wish to elucidate the role of this stronger notion of control, by providing the first examples where it holds true.

Small-time controllability has particularly relevant physical implications, both from a fundamental viewpoint and for technological applications. As a matter of fact, quantum systems, once engineered, suffer of very short lifespan before decaying (e.g., through spontaneous photon emissions) and loosing their non-classical properties (such as superposition). The capability of controlling them in a minimal time is in fact an open challenge also in physics (see, e.g., the pioneering work \cite{khaneja-brockett-glaser} on the minimal control-time for spin systems).

\subsection{Results}
\subsubsection{A small-time approximately controllable equation}

In this paper, we present the first example of small-time approximately controllable Schrödinger equation of the form \eqref{eq:schro}, which is the system
\begin{equation}\label{eq:small-time-oscillator_BIS}
\begin{cases}
i\partial_t\psi(t,x)=
\left(- \Delta + V(x) +u_{1}(t)|x|^2+u_2(t)W_2(x)\right)\psi(t,x), & (t,x) \in (0,T) \times \R^d\,,\\
\psi(0,\cdot)=\psi_0,&
\end{cases}
\end{equation}
where $W_2\in L^{\infty}(\R^d,\R)$ and 
\begin{align}
&V \in L^2_{\rm loc}(\R^d,\R), \qquad 
\exists a,b > 0,\, \forall x \in \R^d, V(x)\geq -a|x|^2-b, \label{Hyp:V_transp}
\\
&\exists c, \delta>0, \gamma>\max\{-2,-d/2\},\,
\forall x \in B_{\R^d}(0,\delta), |V(x)|\leq c|x|^\gamma. \label{Hyp:V}
\end{align}
In the physical literature, the control $u_1$ is usually referred to as \emph{the frequency} of the quadratic potential. For $V=0$, and when $u_1(t)> 0$, this model corresponds to the harmonic oscillator, a system widely investigated both in physics and mathematics. 
It also includes singular potentials $V$, as for instance the Coulomb potential $V(x)= |x|^{-1}$ in dimension $d\geq 3$. 
We prove the following result of global approximate controllability in arbitrarily small times.

\begin{theorem}\label{thm:small-time-global_BIS}
Let $V$ satisfy \eqref{Hyp:V_transp} and (\ref{Hyp:V}). There exists a dense subset $\mathcal{D}$ of $L^\infty(\R^d,\R)$ such that, for every $W_2 \in \mathcal{D}$, system (\ref{eq:small-time-oscillator_BIS}) is small-time $L^2$-approximately controllable.
\end{theorem}

Under appropriate assumptions on $V$ and $W_2$, the small-time $\mathcal{H}$-approximate controllability holds for functional spaces $\mathcal{H}$ more regular than $L^2$ (see Theorem \ref{thm:small-time-global_smooth} where we treat the case $V=0$ for simplicity).

\begin{remark} \label{Rk:ex_W2_ab}
An explicit example of potential $W_2$ for which system (\ref{eq:small-time-oscillator_BIS}) is small-time $L^2$-approximately controllable is, in the case $d=1$, $W_2(x)=e^{ax^2+bx}$, with $a<0$ and $a,b$ algebraically independent (see the end of Section \ref{subsec:General} for a proof).
\end{remark}

\subsubsection{A limiting example}

In Theorem \ref{thm:small-time-global_BIS}, selecting an appropriate potential $W_2$ is required for the small-time approximate controllability, as emphasized by the following counter-example,
\begin{equation}\label{eq:oscillator_BIS}
\begin{cases}
i\partial_t\psi(t,x)=\left(-\Delta
+u_0(t)|x|^2
+ \sum_{j=1}^{d} u_j (t) x_j \right)
\psi(t,x),& (t,x) \in (0,T)\times\R^d, \\
\psi(0,\cdot)=\psi_0.&
\end{cases}
\end{equation}

\begin{theorem} \label{Thm:Lim_Ex}
System (\ref{eq:oscillator_BIS}) is not $L^2$-approximately controllable.
\end{theorem}

Notice that this obstruction holds even in large times. This result is a consequence of a more precise one (cf. Theorem \ref{Prop:Lim_Ex}). A weaker version of it appeared in \cite[Observation II]{teismann}. The subsystem obtained with $u_0=1$ was already known not to be approximately controllable \cite{mirrahimi-rouchon}. Also, the subsystem obtained with $u_0=0$ was already known not to be small-time approximately controllable \cite{beauchard-coron-teismann,beauchard-coron-teismann2}.  Here we show that the additional control $u_0(t)$ does not restore controllability. Additionally, we also provide a description of the approximately reachable set as a product of a fixed number (i.e., independent of the initial and final state) of unitaries (cf. Theorem \ref{Prop:Lim_Ex} (3)).

\subsection{The technique: space-dilations and time-contractions}\label{sec:techniques}

The proof of Theorem \ref{thm:small-time-global_BIS} is based on the idea that, by suitably tuning the quadratic potential $\varphi(x,t)=u(t)|x|^2$, the wavefunction can be controlled to evolve approximately along specific flows, namely the ones generated by the transport operator $\langle\nabla_x\varphi,\nabla\rangle+\frac{1}{2}\Delta_x\varphi=u(t)(2 \langle x,\nabla\rangle+d)$. These are clearly space-dilations 
$$D_\alpha\psi:=\alpha^{d/2}\psi(\alpha x)=\exp\left(\log(\alpha)(\langle x,\nabla\rangle+d/2)\right)\psi.$$
We then use space-dilations, to induce time-contractions: as detailed in Proposition \ref{Prop:exp(isDelta)}, this is obtained by considering the following limit of conjugated dynamics $$e^{is\Delta}\psi =\lim_{t\to 0^+}D_{t^{1/2}}e^{ist(\Delta-V)}D_{t^{-1/2}}\psi.$$
This is basically a time-rescaling of \eqref{eq:small-time-oscillator_BIS}: the important fact is that this time-rescaling can be represented as a composition of three small-time approximate solutions of \eqref{eq:small-time-oscillator_BIS}. On the LHS, $s>0$ can be as large as we want (miming in this way a large-time free evolution), while on the RHS the control time is given by the product $st$, hence we can take $t$ as small as desired to follow a control trajectory in arbitrarily small time: in this way, a large-time free evolution can be approximated by a small-time controlled evolution. This allows us to rescale the time of the drift, and thus transform large-time controllability results into small-time controllability results.

\subsection{Literature review} \label{subsec:Biblio}

There is a vast mathematical literature on the controllability of bilinear Schrödinger equations.  This is also due to the relevance of these mathematical questions for physical and engineering applications. Before discussing the literature on bilinear PDEs, let us briefly mention the controllability properties in the finite-dimensional settings.

\subsubsection{Finite-dimensional systems}
The controllability properties of finite-dimensional bilinear systems of the form 
$$i\frac{d\psi}{dt}=\left(H_0+\sum_{j=1}^mu_j(t)H_j\right)\psi, \quad \psi\in \C^n,$$
where $H_0,...,H_m$ are Hermitian matrices (or, in other words, $iH_0,\dots,iH_m\in \frak{u}(n)$ the Lie algebra of skew-Hermitian matrices), are well-understood. In particular, it is known that exact and approximate controllability are equivalent properties \cite{BGRS}, that large-time controllability (for the unitary propagator evolving in the Lie group $U(n)$) is equivalent to the Lie rank condition ${\rm Lie}\{iH_0,\dots,iH_m\}=\frak{u}(n)$ \cite{sussmann-jurje}, and that small-time controllability is equivalent to the strong Lie rank condition ${\rm Lie}_\psi\{iH_1,\dots,iH_m\}=T_\psi S^{2n-1}$ for some $\psi\in S^{2n-1}:=\{\psi\in\C^n\mid |\psi|=1\}$ (here $T_\psi S^{2n-1}$ denotes the tangent space of the sphere $S^{2n-1}$ at the point $\psi$, and ${\rm Lie}_\psi$ denotes the evaluation of the algebra at the point $\psi$) \cite{time-zero}. We can interpret the latter condition by saying that the drift operator $H_0$ does not contribute to small-time controllability in the finite-dimensional case. This is a fundamental difference w.r.t. the infinite-dimensional setting, where the drift $H_0=-\Delta$ gives a necessary contribution for small-time control: indeed, the strong Lie rank condition is clearly not satisfied in the infinite-dimensional setting, being the $H_j=W_j(x), j=1,\dots,m$, commuting operators of multiplication; the only non-commuting part, in the infinite-dimensional setting, is the ideal generated by the Laplacian $H_0=-\Delta$. This intuitive correspondence between controllability and non-commuting operators is indeed rigorously developed in this paper, in terms of precise choices of non-commuting controlled propagators (for more details we refer to Proposition \ref{Prop:exp(isDelta)} and Theorem \ref{Prop:Lim_Ex}). E.g., notice that the controlled trajectories (we briefly introduced in the Section \ref{sec:techniques}) which follow transport flows, such as space-dilations, are indeed generated by the first order commutator $\frac{1}{2}[\Delta,u(t)|x|^2]:=\frac{u(t)}{2}(\Delta |x|^2-|x|^2\Delta).$

\subsubsection{Topological obstructions to exact controllability}

In \cite{BMS}, Ball, Marsden and Slemrod proved obstructions to global exact controllability of linear PDEs with bilinear controls, such as \eqref{eq:schro}. Precisely, if 
$\mathcal{H}$ is a Hilbert space, subset of $L^2(M,\C)$,
$i(\Delta+V)$ generates a group of bounded operators on $\mathcal{H}$ 
and the multiplicative operators $W_j$ are bounded on $\mathcal{H}$, 
then system \eqref{eq:schro} is not exactly controllable in $\mathcal{H} \cap \mathcal{S}$, with controls 
$u=(u_1,\dots,u_m) \in L^p_{loc}(\R,\R^m)$ with $p>1$.
The fundamental reason behind is that, under these assumptions, the reachable set has empty interior in $\mathcal{H}$.
The case of $L^1_{loc}$-controls ($p=1$) was incorporated in \cite{Chambrion-Caponigro-Boussaid-2020} and extensions to nonlinear equations were proved in \cite{chambrion-laurent,chambrion-laurent2}.

\medskip

Note that this result does not apply to the systems 
\eqref{eq:small-time-oscillator_BIS} and 
\eqref{eq:oscillator_BIS},
with, for instance $\mathcal{H}=L^2(\R^d,\C)$,
because they have unbounded control potentials.


\medskip

After this seminal work \cite{BMS}, different notions of controllability have been studied for system \eqref{eq:schro}, such as exact controllability in more regular spaces (on which the control potentials $W_j$ do not define bounded operators), or approximate controllability: the latter is the one analyzed in this paper. 

\subsubsection{Exact controllability in more regular spaces}

For system \eqref{eq:schro} with $V=0$, $m=1$, $M=(0,1)$ and Dirichlet boundary conditions, local exact controllability was first proved in 
\cite{beauchard1,beauchard-coron} with Nash-Moser techniques, 
to deal with an apparent derivative loss problem,
and then in \cite{beauchard-laurent} with a classical inverse mapping theorem, thanks to a regularizing effect. By grafting other ingredients onto this core strategy, global (resp. local) exact controllability in regular spaces was proved for different models
in \cite{nersesyan-nersisyan,morancey-polarizability} (resp. 
\cite{bournissou}).

\subsubsection{Approximate controllability}

The first results of global approximate controllability of bilinear Schrödinger equations in $L^2$ (and more regular spaces) were obtained in \cite{BCMS,nersesyan,ervedoza}.
More generally, it is known that if the drift $-\Delta+V$ has only point spectrum, \eqref{eq:schro} is globally approximately controllable in large times in $L^2(M)$, generically w.r.t. $V,W_0,\dots,W_m$ \cite{MS-generic}, for a manifold $M$ of arbitrary dimension. The problem of approximately controlling bilinear Schrödinger equations in large time is thus well-understood, if the drift has pure point spectrum only. E.g., in the one-dimensional case, global approximate controllability in large times of \eqref{eq:small-time-oscillator_BIS} with $V=0$ was proved in \cite{BCMS}. On the other hand, harmonic oscillators controlled with linear-in-space potentials only, are highly non-controllable even in large times \cite{mirrahimi-rouchon}; in Theorem \ref{Thm:Lim_Ex} we prove that quadratic-in-space control potentials do not restore controllability. 


In the case of a drift with mixed continuous and point spectrum, few results of large-time approximate controllability (e.g., between bound states) are also available \cite{mirrahimi,chambrion,beauchard-coron-rouchon}. 

\medskip
\subsubsection{Small-time approximate controllability}\label{sec:STAC}
Small-time global approximate controllability has been an open challenge since the early days of bilinear control of Schrödinger equations, both in physics and mathematics. From a mathematical point of view, only few results are available for PDEs. There are known obstructions to the small-time global approximate controllability of (sub)harmonic quantum oscillators \cite{beauchard-coron-teismann,beauchard-coron-teismann2}: these obstructions are related to the approximation of coherent harmonic states and the consequent preservation of Gaussian states for small times, which holds in the presence of linear-in-space control potentials only (see also \cite{obstruction-ivan} for other semi-classical obstructions related to the vanishing set of the $W_j$'s). We also point out an academic example of a small-time globally approximately controllable conservative PDE \cite{boussaid-caponigro-chambrion}, which involves as drift, instead of the Laplacian, the operator $|\Delta|^\alpha,\alpha> 5/2$, on the 1-D torus $M=\T$.

Recently, a renewed interest has been brought by the work \cite{duca-nersesyan}, where small-time bilinear approximate controllability between eigenstates is proved on the $d$-dimensional torus $\T^d$, by means of an infinite-dimensional geometric control approach, adapted to bilinear Schrödinger equations. This approach was firstly developed for the additive control of Navier-Stokes equations \cite{agrachev-sarychev,agrachev2}. Several works on the small-time approximate controllability of Schrödinger and wave equations, exploiting similar techniques, have followed \cite{coron-xiang-zhang,chambrion-pozzoli,duca-pozzoli,pozzoli,Boscain-2024}.

With respect to this literature, Theorem \ref{thm:small-time-global_BIS} is the first available result, on bilinear Schrödinger equations, of global approximate controllability in arbitrarily small times. From a technical point of view, instead of exploiting fast local phase control as in \cite{duca-nersesyan}, we build controlled trajectories following arbitrarily fast transport flows: this technique will be generalized to other potentials (i.e., not necessarily quadratic) in a future work. 



\subsection{Structure of the paper}

In Section \ref{sec:preliminaries} we introduce some control notions and recall some functional analytic tools.  
In Section \ref{Sec:Ex2}, we prove Theorem \ref{thm:small-time-global_BIS}. 
In Section \ref{sec:Lim_Ex}, we prove Theorem \ref{Thm:Lim_Ex}.
In Section \ref{sec:TM_reg}, we adapt Theorem \ref{thm:small-time-global_BIS} to higher regularities.


\section{Preliminaries}\label{sec:preliminaries}

In this section, we introduce the set of small-time reachable operators, prove its semi-group structure and its closure.
We also recall some classical tools of functional analysis that we shall need in the rest of the paper.

\subsection{Small-time approximately reachable operators}
\label{Subsec:STAROp}

To ensure well-posedness and facilitate our control strategy, we introduce a notion of admissible controls.

\begin{definition}[Set of admissible controls]\label{Def:adm_cont}
    Let $(\mathcal{H},\|.\|_{\mathcal{H}})$ be a normed $\mathbb{C}$-vector space, subset of $L^2(M,\C)$.  $\mathcal{U}$ is a set of admissible controls for the system \eqref{eq:schro} and the state space $\mathcal{H}$ if
\begin{itemize}
\item 
for every $T>0$, 
$u \in \mathcal{U}(0,T)$ and $\psi_0 \in \mathcal{H} \cap \mathcal{S}$,
\eqref{eq:schro} has a unique solution $\psi \in C^0([0,T],\mathcal{H})$ and $\psi(T;u,.)$ is a bounded operator on $\mathcal{H}$,
\item 
for every 
$T_1, T_2>0$, 
$ u_1 \in \mathcal{U}(0,T_1)$ and $u_2 \in \mathcal{U}(0,T_2)$ there holds
$u_1 \sharp u_2 \in \mathcal{U}(0,T_1+T_2)$, where, for every $t \in (0,T_1+T_2)$,
$$u_1 \sharp u_2(t)=
\left\lbrace \begin{array}{ll}
u_1(t) \quad & \text{ if } t \in (0,T_1), \\
u_2(t-T_1) & \text{ if } t \in (T_1,T_1+T_2).
\end{array}\right.$$
\end{itemize}
\end{definition}

In this article, most often
$\mathcal{H}=L^2(\R^d,\C)$
(except for Sections \ref{sec:TM_reg})
and $\mathcal{U}$ is the space of piecewise constant functions 
(except for Sections \ref{Subsec:Toy} and \ref{sec:TM_reg}),
which is clearly stable by concatenation.
Small-time approximately reachable states will be described in terms of small-time approximately reachable operators.

\begin{definition}[Small-time $\mathcal{H}$-approximately reachable operator, $\mathcal{H}$-STAR] \label{Def:STAR}
 Let $(\mathcal{H},\|.\|_{\mathcal{H}})$ be a normed $\C$-vector
 space, subset of $L^2(M,\C)$ and $\mathcal{U}$ be a set of admissible controls for the system \eqref{eq:schro} and the state space $\mathcal{H}$. Let also $\mathcal{L}(\mathcal{H})$ be the algebra of bounded linear operators acting on $\mathcal{H}$.
\begin{itemize}
\item For $T>0$, an isometry $L \in \mathcal{L}(\mathcal{H})$ is \textbf{$\mathcal{H}$-approximately  (resp. $\mathcal{H}$-exactly) reachable in time $T$} if, for every $\psi_0 \in \mathcal{H} \cap \mathcal{S}$ and $\varepsilon>0$, there exists $u \in \mathcal{U}(0,T)$ such that 
 $\| \psi(T;u,\psi_0)-  L \psi_0 \|_{\mathcal{H}}<\varepsilon$ (resp. $\psi(T;u,\psi_0)= L \psi_0$).

\item Given $T \geq 0$, the operator $L$ is \textbf{$\mathcal{H}$-approximately (resp. $\mathcal{H}$-exactly) reachable in time $T^+$} if,  for every $\psi_0 \in \mathcal{H} \cap \mathcal{S}$ and $\varepsilon>0$, there exists 
$T_1 \in [T,T+\varepsilon]$ such that it is $\mathcal{H}$-approximately (resp. $\mathcal{H}$-exactly) reachable in time $T_1$.


\item The operator $L$ is \textbf{small-time $\mathcal{H}$-approximately  (resp. $\mathcal{H}$-exactly) reachable} if it is $\mathcal{H}$-approximately  (resp. $\mathcal{H}$-exactly) reachable in time $0^+$. Then, we use the abbreviation: \textquotedblleft the operator $L$ is $\mathcal{H}$-STAR \textquotedblright.
\end{itemize}
\end{definition}

For instance, let us consider system (\ref{eq:small-time-oscillator_BIS}), 
with state space $\mathcal{H}=L^2(\R^d,\C)$. If the set of constant controls is admissible then, for every $\alpha_1, \alpha_2 \in \R$ and $\sigma>0$, 
the operator $e^{i\sigma(\Delta-V-\alpha_1 |x|^2-\alpha_2 W_2)}$
is $L^2$-exactly reachable in time $\sigma$. 
Indeed, if $u=(u_1,u_2):(0,\sigma) \rightarrow \R^2$ is the constant function $(\alpha_1,\alpha_2)$ then, for every $\psi_0 \in L^2(\T^d,\C)$, we have $\psi(\sigma;u,\psi_0)=e^{i\sigma(\Delta-V-\alpha_1 |x|^2-\alpha_2 W_2)}\psi_0$.

\begin{lemma}\label{lem:reachable-operators}
\begin{enumerate}
\item If $L_1$ is $\mathcal{H}$-approximately (or $\mathcal{H}$-exactly) reachable in time $T_1$, and $L_2$ is $\mathcal{H}$-approximately (or $\mathcal{H}$-exactly) reachable in time $T_2$, then $L_2 L_1$ is $\mathcal{H}$-approximately (or $\mathcal{H}$-exactly) reachable in time $T_1+T_2$. In particular, if $L_1,L_2$ are $\mathcal{H}$-STAR, then $L_2L_1$ also is.  
\item Let $T \geq 0$ and $(L_n)_{n\in\N}$ be a sequence of operators that are $\mathcal{H}$-approximately reachable in time $T^+$ and $L \in \mathcal{L}(\mathcal{H})$ such that, for every $\psi \in \mathcal{H}$, 
$\|(L_n -L)\psi\|_{\mathcal{H}} \rightarrow 0$. Then the operator $L$ is $\mathcal{H}$-approximately reachable in time $T^+$.
\item The set of $\mathcal{H}$-STAR operators is a subsemigroup 
of $\mathcal{L}(\mathcal{H}) \cap \text{Isom}(L^2(M,\C))$ closed for the topology of the strong convergence.
\end{enumerate}
\end{lemma}

\begin{proof}
\emph{1.} Let $\psi_0 \in \mathcal{H} \cap \mathcal{S}$ and $\varepsilon>0$. There exists $u_2 \in \mathcal{U}(0,T_2)$ such that
\begin{equation} \label{def:T2u2}
\| \psi(T_2; u_2, L_1 \psi_0 ) -  L_2 L_1 \psi_0 \|_{\mathcal{H}}< \frac{\varepsilon}{2}.
\end{equation}
There exists $u_1 \in \mathcal{U}(0,T_1)$ such that
\begin{equation} \label{def:T1u1}
\| \psi(T_1;u_1,\psi_0)- L_1 \psi_0 \|_{\mathcal{H}} < \frac{\varepsilon}{2 \| \psi(T_2;u_2,.)\|_{\mathcal{L}(\mathcal{H})}}.
\end{equation}
Then 
$u:=u_1 \sharp u_2 \in \mathcal{U}(0,T_1+T_2)$ 
by Definition \ref{Def:adm_cont}.
Moreover, by using the triangular inequality, Definition \ref{Def:adm_cont}, (\ref{def:T2u2}) and (\ref{def:T1u1}), we obtain
$$\begin{aligned}
& \| \psi(T_1+T_2;u,\psi_0) -  L_2 L_1 \psi_0 \|_{\mathcal{H}} 
\\ \leq & \| \psi(T_2;u_2,\psi(T_1;u_1,\psi_0))- \psi(T_2;u_2, L_1 \psi_0) \|_{\mathcal{H}} + \| \psi(T_2;u_2, L_1 \psi_0)-   L_2 L_1 \psi_0\|_{\mathcal{H}}
\\  \leq & \| \psi(T_2;u_2,.) \|_{\mathcal{L}(\mathcal{H})} \| \psi(T_1;u_1,\psi_0)-  L_1 \psi_0 \|_{\mathcal{H}} +  
\| \psi(T_2;u_2, L_1 \psi_0)-   L_2 L_1 \psi_0\|_{\mathcal{H}}
<  \varepsilon.
\end{aligned}$$
\noindent \emph{2.} Let $\psi_0 \in \mathcal{H} \cap \mathcal{S}$ and $\varepsilon>0$. There exists $n \in \N$ such that 
$\|(L_n -L)\psi_0\|_{\mathcal{H}} < \varepsilon/2$. There exists $T_1 \in [T,T+\epsilon]$, and $u\in\mathcal{U}(0,T_1)$ such that $\| \psi(T_1;u,\psi_0)- L_n \psi_0 \|_{\mathcal{H}} < \varepsilon/2$. Then, by triangular inequality
$$
\| \psi(T_1;u,\psi_0)-L \psi_0 \|_{\mathcal{H}}
\leq \| \psi(T_1;u,\psi_0)- L_n \psi_0 \|_{\mathcal{H}} + \| (L_n -L)\psi_0\|_{\mathcal{H}} < \varepsilon.
$$
\noindent \emph{3.} is a consequence of \emph{1.} and \emph{2.}

\end{proof}
    
This Lemma proves that the composition of an arbitrary number
(resp. the strong limit) of $\mathcal{H}$-STAR operators is a $\mathcal{H}$-STAR operator. These basic facts will be massively used in this article.

\subsection{Well posedness for piecewise constant controls}
\label{Subsec:WP_pwcst}

\begin{proposition}\label{prop:self-adjointness}{\cite[Corollary page 199]{rs2}}
Let 
$V$ satisfying (\ref{Hyp:V_transp}).
Then 
$-\Delta+V$
is essentially self-adjoint on $C^\infty_c(\R^d,\C)$.
\end{proposition}

\begin{definition}
Given two densely defined linear operators $A$ and $B$ with domains $D(A)$ and $D(B)$ on an Hilbert space $\mathcal{H}$, $B$ is said to be $A$-bounded if $D(A)\subset D(B)$ and
there exist $a,b\geq 0$ such that for all $\psi\in D(A)$
$$\|B\psi\|<a\|A\psi\|+b\|\psi\|.$$
The infimum of such $a$ is called the relative bound of $B$.
\end{definition}

\begin{proposition}\label{prop:kato-rellich}{(Kato-Rellich Theorem)\cite[Theorem X.12]{rs2}}
If $A$ is self-adjoint and $B$ is symmetric and $A$-bounded with relative bound $a<1$, then $A+B$ is self-adjoint on $D(A)$ and essentially self-adjoint on any core of $A$.
\end{proposition}


In light of Propositions \ref{prop:self-adjointness} and
\ref{prop:kato-rellich}, we can define the solutions 
of the systems 
\eqref{eq:small-time-oscillator_BIS},
\eqref{eq:oscillator_BIS},
associated with piecewise constant controls, by composition of time-independent unitary propagators associated with self-adjoint operators (see, e.g., \cite[Definition p.256 \& Theorem VIII.7]{rs1}). For instance, for system \eqref{eq:small-time-oscillator_BIS}, 
given a subdivision $0=t_0<\dots<t_N=T$,
a piecewise constant control 
$u:[0,T]\to \mathbb{R}^{2}$ defined as 
$u(t)=(u^{j}_1,u^{j}_2)\in\mathbb{R}^{2}$ when 
$t \in [t_{j-1},t_j]$, 
and an initial condition $\psi_0 \in L^2(\R^d,\C)$,
the solution $\psi\in C^0([0,T],L^2(\R^d,\mathbb{C}))$ of \eqref{eq:small-time-oscillator} is defined by
\begin{equation}\label{eq:propagator}
\psi(t;u,\psi_0)=
e^{i(t-t_{j-1})(\Delta-V-u_1^{j}|x|^2-u_2^j W_2)}
e^{i \tau_{j-1}(\Delta-V-u_1^{j-1} |x|^2-u_2^{j-1}W_2)}
\dots
e^{i\tau_1 (\Delta-V-u_1^1|x|^2-u_2^1 W_2)} \psi_0.
\end{equation}
where $\tau_l=(t_l-t_{l-1})$ for $l=1,\dots,N$.


\subsection{A conjugation formula}

\begin{proposition}\label{lem:conjugation}
Let $A,B$ be self-adjoint operators on the Hilbert space $\mathcal{H}$, and suppose that $e^{iB}$ is an isomorphism of $D$, where $D$ is a core for $A$. Then, for any $t\in \R$,
$$e^{-iB}e^{itA}e^{iB} =\exp(e^{-iB}itAe^{iB}).$$
\end{proposition}
\begin{proof}
Since $e^{iB}:D\to D$ is an isomorphism, $L:=e^{-iB}Ae^{iB}$ is essentially self-adjoint on $D$, hence, $e^{itL}:=\exp(e^{-iB}itAe^{iB})$ is well-defined for $t\in\R$. For $\psi_0 \in D$,
$$
\psi(t):=e^{itL}\psi_0 
\quad \text{ and } \quad 
\Psi(t):=e^{ iB}\psi(t)
$$
solve
$$
\begin{cases}
\frac{d}{dt}\psi(t)=i L\psi(t), \\
 \psi(0)=\psi_0,&
\end{cases}
\quad \text{ and } \quad
\begin{cases}
\frac{d}{dt}\Psi(t)=iA\,\Psi(t), &\\
 \Psi(0)=e^{ iB}\psi_0,&
\end{cases}
 $$
thus $\Psi(t)=e^{it A} e^{iB} \psi_0$ and
$e^{itL} \psi_0=e^{-iB} e^{it A} e^{iB} \psi_0$.
This holds for every $\psi_0 \in D$, thus $e^{itL}=e^{-iB} e^{it A} e^{iB}$.
\end{proof}


\subsection{A convergence property}

\begin{proposition}{\cite[Theorem VIII.21 \& Theorem VIII.25(a)]{rs1}}\label{prop:trotter}
Let $(A_n)_{n\in\N},A$ be self-adjoint operators on an Hilbert space $\mathcal{H}$, with a common core $D$. If $\| (A_n-A)\psi \|_{\mathcal{H}} \underset{n \to\infty}{\longrightarrow} 0$ for any $\psi\in D$, then 
$\| (e^{iA_n}-e^{iA})\psi \|_{\mathcal{H}}\underset{n \to\infty}{\longrightarrow} 0$ for any $\psi\in \mathcal{H}$.
\end{proposition}

\subsection{Trotter-Kato product formula}

\begin{proposition}{(Trotter-Kato product formula) \cite[Theorem VIII.30]{rs1}}\label{prop:trotter-kato}
Let $A,B$ be self-adjoint operators on a Hilbert space $\mathcal{H}$ such that $A+B$ is self-adjoint on $D(A)\cap D(B)$. Then, for every $\psi_0 \in \mathcal{H}$,
$$ \left\| \left(e^{i\frac{A}{n}}e^{i\frac{B}{n}}\right)^n\psi_0 - e^{i(A+B)}\psi_0\right\|_{\mathcal{H}} \underset{n \to + \infty}{\longrightarrow} 0.$$
\end{proposition} \label{Prop:ESA+WP}

\section{Small-time approximately controllable examples} \label{Sec:Ex2}

In this section, we prove Theorem \ref{thm:small-time-global_BIS}. 
We work with the state space $\mathcal{H}=L^2(\R^d,\C)$ and the set of admissible controls made of piecewise constant functions $u=(u_1,u_2):(0,T)\rightarrow \R^2$ (see Section \ref{Subsec:WP_pwcst}).

In Section \ref{Subsec:LTAC}, we present the proof strategy: thanks to a well-known result of large-time approximate controllability for system (\ref{eq:small-time-oscillator_BIS}) with $V=0$, it suffices to prove that, for every $\sigma \geq 0$, the operator $e^{i\sigma(\Delta-|x|^2)}$ 
is $L^2$-STAR.

In Section \ref{Subsec:Toy}, we first study the particular case $V=0$, for which
the operators $e^{i\sigma(\Delta-|x|^2)}, \sigma\geq 0,$ are small-time \textbf{exactly} reachable. The proof is made particularly simple by an explicit representation formula of the solutions, inspired by \cite[Section 4]{carles}.

In Section \ref{subsec:General}, we further prove in the general case (i.e. for any $V$ satisfying \eqref{Hyp:V_transp} and (\ref{Hyp:V})) that the operators $e^{i\sigma(\Delta-|x|^2)}, \sigma\geq0,$ are $L^2$-STAR. Here, the proof is more sophisticated: it relies on some small-time limits of conjugated dynamics (given in the proof of Proposition \ref{Prop:exp(isDelta)}).

\subsection{Proof strategy relying on large time approximate control} \label{Subsec:LTAC}
\begin{proof}[Proof of Theorem \ref{thm:small-time-global_BIS}]
The large-time $L^2$-approximate controllability of system (\ref{eq:small-time-oscillator_BIS}) with $V=0$ is known to hold for a dense subset of potentials $W_2 \in L^\infty(\R^d,\R)$, see \cite[Theorem 2.6]{BCCS} and \cite[Proposition 4.6]{MS-generic}. 

\begin{proposition} \label{Prop:Boscain&al}
The system (\ref{eq:small-time-oscillator_BIS}) with $V=0$ is large time $L^2$-approximately controllable, generically with respect to $W_2 \in L^{\infty}(\R^d,\R)$. More precisely, there exists a dense subset $\mathcal{D}$ of $L^{\infty}(\R^d,\R)$ such that, for every $W_2 \in \mathcal{D}$ and $\psi_0 \in \mathcal{S}$, the set
$$
\{e^{i \sigma_k( \Delta- |x|^2+\alpha_k W_2) } \dots  
e^{i \sigma_1(\Delta-|x|^2+\alpha_1 W_2 )}\psi_0\, ;\, k \in \N^*, \sigma_1,\dots,\sigma_k\geq 0, \alpha_1,\dots,\alpha_k \in\R \}$$
is dense in $\mathcal{S}$.
\end{proposition}

Thus, by taking into account the semigroup structure of the set of $L^2$-STAR operators (Lemma \ref{lem:reachable-operators}),
in order to prove Theorem \ref{thm:small-time-global_BIS},
it suffices to prove the following result.

\begin{proposition} \label{Prop:Approx_BIS}
We assume $V$ satisfies \eqref{Hyp:V_transp} and (\ref{Hyp:V}), and $W_2 \in L^{\infty}(\R^d,\R)$.
System (\ref{eq:small-time-oscillator_BIS}) satisfies the following property: for every
$\sigma \geq 0$ and
$\alpha \in \R$, the operator
$e^{i \sigma(\Delta-|x|^2+\alpha W_2)}$
is $L^2$-STAR.
\end{proposition}

To obtain Proposition \ref{Prop:Approx_BIS}, all we have to do is demonstrate it for $\alpha=0$.

\begin{proposition} \label{Prop:Approx_BIS_alpha=0}
We assume $V$ satisfies \eqref{Hyp:V_transp} and (\ref{Hyp:V}), and $W_2 \in L^{\infty}(\R^d,\R)$.
System (\ref{eq:small-time-oscillator_BIS}) satisfies the following property: for every
$\sigma \geq 0$
the operator $e^{i \sigma(\Delta-|x|^2)}$
is $L^2$-STAR.
\end{proposition}

\begin{proof}[Proof of Proposition \ref{Prop:Approx_BIS} thanks to Proposition \ref{Prop:Approx_BIS_alpha=0}:] 
Let $\sigma\geq 0$ and $\alpha\in\R$.

\medskip

\noindent \emph{Step 1: We prove that the operator $e^{i\frac{\sigma\alpha}{n} W_2}$ is $L^2$-STAR for any $n\in \N$.}
Given $\tau>0$, the operator $L_\tau=e^{i\tau(\Delta-V+\frac{\sigma\alpha}{\tau n}W_2)}$ is exactly reachable in time $\tau$ because associated with the constant control $(u_1,u_2)=(0,-\frac{\sigma\alpha}{\tau n})$. Moreover, by Proposition \ref{prop:trotter}, for every $\psi\in L^2(\R^d,\C)$, $\|L_\tau\psi-e^{i\frac{\sigma\alpha}{n} W_2}\psi\|_{L^2}\to 0$ as $\tau\to 0$. By Lemma \ref{lem:reachable-operators}, the operator $e^{i\frac{\sigma\alpha}{n} W_2}$ is $L^2$-STAR.

\medskip
\noindent \emph{Step 2: We use the Trotter-Kato formula.} The operator $A := \sigma( \Delta  - |x|^2)$ is self-adjoint on
$D(A):=\{ \psi \in H^2(\R^d,\C) ; |x|^2 \psi \in L^2(\R^d,\C) \}$. 
The operator $B:={\alpha \sigma} W_2$ is self-adjoint on $L^2(\R^d,\C)$ because $W_2 \in L^\infty(\R^d,\R)$. The operator $A+B=\sigma(\Delta-|x|^2+\alpha W_2)$ is self-adjoint on $D(A)\cap D(B)=D(A)$ thanks to Proposition \ref{prop:kato-rellich}. We consider the operator 
$$L_n=\left(e^{i\frac{\sigma}{n}(\Delta-|x|^2)}e^{i\frac{\alpha\sigma}{n}W_2}\right)^n, $$
which is $L^2$-STAR thanks to the hypothesis, Step 1, and Lemma \ref{lem:reachable-operators}. By Proposition \ref{prop:trotter-kato}, for every $\psi \in L^2(\R^d,\C)$,
$\| (L_n-e^{i \sigma(\Delta-|x|^2+\alpha W_2)}) \psi \|_{L^2} \to 0$ as $n \to \infty$.
Thus, by Lemma \ref{lem:reachable-operators}, the operator
$e^{i\sigma (\Delta-|x|^2+ \alpha W_2)}$ is $L^2$-STAR.
\end{proof}
The proof of Theorem \ref{thm:small-time-global_BIS} then follows if we show Proposition \ref{Prop:Approx_BIS_alpha=0}.
\end{proof}

\subsection{Toy model $V=0$: an explicit representation formula} \label{Subsec:Toy}

In this section, we prove Proposition \ref{Prop:Approx_BIS_alpha=0} in the particular case $V=0$, i.e. for the system
\begin{equation}\label{eq:small-time-oscillator}
\begin{cases}
i\partial_t\psi(t,x)=\left(- \Delta
+u_{1}(t)|x|^2+u_2(t)W_2(x)\right)\psi(t,x), & (t,x) \in (0,T) \times \R^d,\\
\psi(0,\cdot)=\psi_0. &
\end{cases}
\end{equation}
In Section \ref{Subsec:WP}, we introduce a set $\mathcal{U}(0,T)$ of admissible controls for system (\ref{eq:small-time-oscillator}), which is slightly larger than the space of piecewise constant functions $u=(u_1,u_2):(0,T) \rightarrow \R^2$, because it will be more practical.
In Section \ref{Subsec:ExactST}, thanks to an explicit representation formula of the solutions, inspired by \cite[Section 4]{carles}, we prove that, for every $\sigma \geq 0$, the operator $e^{i\sigma(\Delta-|x|^2)}$ is small-time exactly reachable, which implies Proposition \ref{Prop:Approx_BIS_alpha=0} when $V=0$.

\subsubsection{Admissible controls} \label{Subsec:WP}

\begin{definition} \label{Def:Adm_cont}
Let $T>0$ and
$u_1, u_2:[0,T] \rightarrow \R$.
$(u_1,u_2) \in \mathcal{U}(0,T)$ if
$u_2$ is piecewise constant on $[0,T]$,
$u_1$ is piecewise constant on $\{t \in [0,T];u_2(t) \neq 0\}$, measurable and uniformly bounded on $\{ t \in [0,T]; u_2(t)=0\}$.
\end{definition}
\begin{proposition} \label{Prop:WP}
Let 
$W_2 \in L^{\infty}(\R^d,\R)$,
$T>0$, 
$(u_1, u_2) \in \mathcal{U}(0,T)$, 
and 
$\psi_0 \in L^2(\R^d,\C)$. 
The Cauchy problem (\ref{eq:small-time-oscillator}) has a unique solution 
$\psi \in C^0([0,T],L^2(\R^d,\C))$ and
$\| \psi(.)\|_{L^2}=\|\psi_0\|_{L^2}$.
\end{proposition}
\begin{proof}
\noindent \emph{Step 1: On a time interval $[t_1,t_2]$ on which $u_1$ and $u_2$ are constant.} Proposition \ref{prop:self-adjointness} proves that
$(e^{i \tau (\Delta - u_1 |x|^2 - u_2 W_2)})_{\tau \in \R}$ is a group of unitary operators on $L^2(\R^d,\C)$.

\medskip

\noindent \emph{Step 2: On a time interval $[t_1,t_2]$ on which $u_2=0$.} The time dependent potential $V(t,x)=u_1(t)|x|^2$ satisfies the assumptions (V.I) and (V.II) of \cite[p.1]{fujiwara}, i.e. 
for almost every $t \in [t_1,t_2]$, 
$V(t,.) \in C^\infty(\R^d;\R)$, 
$V \in L^\infty((t_1,t_2)\times B_{\R^d}(0,1))$ and
for every $\alpha \in \N^d$ such that $|\alpha| \geq 2$ then
$\partial_x^{\alpha} V \in L^{\infty}((t_1,t_2)\times\R^d)$.
Thus, by \cite[Theorem 3]{fujiwara}, the equation
$i \partial_t \psi = (-\Delta+V(t,x)) \psi$ has a well defined propagator $(U(t,s))_{t_1 \leq s \leq t \leq t_2}$ of bounded operators on $L^2(\R^d,\C)$. 

\medskip

\noindent The conclusion of Proposition \ref{Prop:WP} is obtained by concatenating these propagators.
\end{proof}

The space $\mathcal{U}$ is clearly stable by concatenation, 
thus it is a set of admissible controls for 
system (\ref{eq:small-time-oscillator})
and the state space $\mathcal{H}=L^2(\R^d,\C)$,
in the sense of Definition \ref{Def:adm_cont} .

\subsubsection{Small-time exact reachability of $e^{i \sigma(\Delta-|x|^2)}$} \label{Subsec:ExactST}

In this section, we focus on the subsystem
\begin{equation}\label{eq:trap}
\begin{cases}
i\partial_t\psi(t,x)=(- \Delta + u(t)|x|^2)\psi(t,x),& (t,x) \in (0,T)\times\R^d, \\
\psi(0,\cdot)=\psi_0.&
\end{cases}
\end{equation}
The state space is $\mathcal{H}=L^2(\R^d,\C)$ and the set of admissible controls is $L^\infty((0,T),\R)$.

\begin{definition} [Dilations]
For $\alpha \in \R^*$, the dilation $D_{\alpha}$
is the linear isometry of $L^2(\R^d,\C)$ defined by $(D_{\alpha}\psi)(x)=|\alpha|^{d/2}\psi(\alpha x)$.
\end{definition}

The case $V=0$ of Proposition \ref{Prop:Approx_BIS_alpha=0} is a consequence of the following result.

\begin{proposition} \label{Prop_exact}
System \eqref{eq:trap} satisfies the following property: 
for every $\sigma \geq 0$, the operator
$e^{i \sigma(\Delta-|x|^2)}$ is small-time $L^2$-exactly reachable.
\end{proposition}

\begin{proof}[Proof of Proposition \ref{Prop_exact}]

\noindent \emph{Step 1: An explicit representation formula for (\ref{eq:trap}).}
Let $T>0$ and $u \in L^\infty(0,T)$.
We assume that the maximal solution $(a, b, \zeta)$ of the differential equation
\begin{equation} \label{eq:abxi}
\left\lbrace \begin{array}{l}
\dot{a}(t) = - 4 a(t)^2  + \frac{1}{b(t)^4} - u(t), \\
\dot{b}(t)= 4 a(t) b(t), \\
\dot{\zeta}(t)=\frac{1}{b(t)^2}, \\
(a,b,\zeta)(0)=(0,1,0),
\end{array}\right.
\end{equation}
is defined on $[0,T]$.
Then, basic computations prove that,
for every $t \in [0,T]$ (see Appendix \ref{Appendix1} for details)
\begin{equation} \label{def:b_zeta}
b(t)=e^{4 \int_0^t a(s) ds}>0, \qquad 
\zeta(t)=\int_0^t e^{- 8 \int_0^{s} a } ds,
\end{equation}
\begin{equation} \label{explicit}
\psi(t;u,.)=e^{ia(t) |x|^2} D_{\frac{1}{b(t)}} e^{i\zeta(t)(\Delta-|x|^2)}\psi_0.
\end{equation}


\medskip

\noindent \emph{Step 2: Given $\sigma,T >0$, we prove there exists $u \in L^{\infty}(0,T)$ such that the solutions of (\ref{eq:abxi}) are defined on $[0,T]$ and satisfy 
$(a,b,\zeta)(T)=(0,1,\sigma)$.}
Indeed, there exists $f \in C^\infty_c((0,1),\R)$ such that
$$\int_0^1 e^{-8 f(s)} ds = \frac{\sigma}{T}.$$
Then the functions $a,b,\zeta,u:[0,T] \rightarrow \R$, defined by
$$a(t):= \frac{1}{T} f' (t/T), \quad 
b(t):=e^{4 f(t/T)}, \quad 
\zeta(t):=\int_0^t e^{-8 f(s/T)} ds, \quad
u(t):=-\dot{a}(t) - 4 a(t)^2  + \frac{1}{b(t)^4} $$
give the conclusion.
\end{proof}

\subsection{General case: small-time approximate reachability of $e^{i \sigma(\Delta-|x|^2)}$} \label{subsec:General}

In this section, we focus on the subsystem
\begin{equation}\label{eq:trap_VV}
\begin{cases}
i\partial_t\psi(t,x)=(- \Delta + V(x) + u(t)|x|^2)\psi(t,x),& (t,x) \in (0,T)\times\R^d, \\
\psi(0,\cdot)=\psi_0.&
\end{cases}
\end{equation}
where $V$ satisfies \eqref{Hyp:V_transp} and (\ref{Hyp:V}), and $u:(0,T) \rightarrow \R$ is piecewise constant. Clearly, Proposition \ref{Prop:Approx_BIS_alpha=0} is a consequence of the following result.

\begin{proposition} \label{Prop:exp(isDelta)}
We assume $V$ satisfies \eqref{Hyp:V_transp} and (\ref{Hyp:V}).
For system (\ref{eq:trap_VV}), the following operators are $L^2$-STAR:
\begin{enumerate}
    \item $e^{i \delta |x|^2}$ for every $\delta\in\R$,
    \item $D_{\alpha}$ for every $\alpha > 0$,
    \item $e^{i \sigma \Delta}$ for every $\sigma \geq 0$,
    \item $e^{i \sigma (\Delta-|x|^2)}$ for every $\sigma \geq 0$.
    \end{enumerate}
\end{proposition}

\begin{proof}

The notations used to prove the Item $i$ are not valid in the proof of the Item $j \neq i$.

\medskip

\noindent \emph{1.} Let $\delta\in\R^*$. For $\tau>0$, the operator $\tau (\Delta-V)+ \delta |x|^2$ is essentially self-adjoint on $C^{\infty}_c(\R^d,\C)$ by Proposition \ref{prop:self-adjointness}, thus its closure $A_{\tau}$ is self-adjoint. 
The operator $A_{0}:=\delta |x|^2$ is self-adjoint on $D(A_0):=\{ \psi \in L^2(\R^d,\C) ; |x|^2 \psi \in L^2(\R^d,\C) \}$. $C^{\infty}_c$ is a common core to $A_0$ and $A_{\tau}$. For any $\psi \in C^{\infty}_c(\R^d,\C)$,
$\|(A_{\tau}-A_0)\psi\|_{L^2} = \tau \| (\Delta-V) \psi \|_{L^2} \to 0$ as $\tau \to 0$. Thus, by Proposition \ref{prop:trotter}, for every
$\psi \in L^2(\R^d,\C)$ 
$\|e^{i A_{\tau}} \psi - e^{i A_0} \psi \|_{L^2} \rightarrow 0$ as $\tau \to 0$. 
Moreover, for every $\tau>0$, the operator $e^{i A_{\tau}}$ is $L^2$-exactly reachable in time $\tau$ because associated with the constant control $u=-\delta/\tau$.
Thus, by Lemma \ref{lem:reachable-operators}, the operator 
$e^{i A_0}$ is $L^2$-STAR.

\medskip

\noindent \emph{2.} Let $\alpha>0$. For every $\tau>0$, the operator
$$L_{\tau} := e^{i\frac{\log(\alpha) |x|^2}{4\tau}}
e^{i\tau\left(\Delta-V+\frac{\log^2(\alpha) |x|^2}{4\tau^2}\right)}
e^{-i\frac{\log(\alpha) |x|^2}{4\tau}}$$
is $L^2$ approximately reachable in time $\tau^+$, 
by Lemma \ref{lem:reachable-operators} and Item 1.

\medskip

\noindent \emph{Step 1: We prove that, for every $\tau>0$,
$$L_{\tau} = \exp\left( i\tau(\Delta-V)+\frac{\log(\alpha)}{2}\left(d+2\la x,\nabla\ra\right)
 \right). $$}
The operator $\tau (\Delta-V+ \log^2(\alpha) |x|^2/ (4\tau^2))$ is essentially self-adjoint on $C^\infty_c(\R^d,\C)$ by Proposition \ref{prop:self-adjointness}, thus its closure $A_{\tau}$ is self-adjoint. The operator $B=\log(\alpha) |x|^2 / (4\tau)$ is self-adjoint on $D(B):=\{ \psi \in L^2(\R^d,\C) ; |x|^2 \psi \in L^2(\R^d,\C) \}$. $C^{\infty}_c(\R^d,\C)$ is a core of $A_{\tau}$ and $e^{iB}$ is an isomorphism of $C^{\infty}_c(\R^d,\C)$. Thus, by Proposition \ref{lem:conjugation},
$$L_{\tau}= \exp \left( i\tau
e^{i\frac{\log(\alpha) |x|^2}{4\tau}}
\left(\Delta-V+\frac{\log^2(\alpha) |x|^2}{4\tau^2}\right)
e^{-i\frac{\log(\alpha) |x|^2}{4\tau}}
\right).$$
Moreover, standard computations  prove
\begin{equation} \label{log(Ltau)}
i\tau
e^{i\frac{\log(\alpha) |x|^2}{4\tau}}
\left(\Delta-V+\frac{\log^2(\alpha) |x|^2}{4\tau^2}\right)
e^{-i\frac{\log(\alpha) |x|^2}{4\tau}}
=
i \tau(\Delta-V)+\frac{\log(\alpha)}{2}\left(d+2\la x,\nabla\ra\right).
\end{equation}

\medskip

\noindent \emph{Step 2: We prove that, for every $\psi \in L^2(\R^d,\C)$, $\|(L_{\tau}-D_{\alpha})\psi\|_{L^2} \to 0$ as 
$\tau \to 0$.}
By (\ref{log(Ltau)}) and Proposition \ref{prop:self-adjointness},
for $\tau>0$,
the operator $\tau(\Delta-V)-i (d+2\la x,\nabla\ra)\log(\alpha)/2$ is essentially self-adjoint on $C^\infty_c(\R^d,\C)$, thus its closure $B_{\tau}$ is self-adjoint. The operator 
$B_0:=-i (d+2\la x,\nabla\ra)\log(\alpha)/2$ is self-adjoint on $D(B_0):=\{ \psi \in L^2(\R^d,\C) ; \langle x , \nabla \psi \rangle \in L^2(\R^d,\C)  \}$. $C^{\infty}_c(\R^d,\C)$ is a common core of $B_{\tau}$ and $B_0$. For every $\psi \in C^\infty_c(\R^d,\C)$,
$\| (B_{\tau}-B_0) \psi \|_{L^2} = \tau \|(\Delta-V) \psi \|_{L^2} \to 0$ as $\tau \to 0$. 
Thus, by Proposition \ref{prop:trotter} and Step 1, for every $\psi \in L^2(\R^d,\C)$,
$$\| (L_{\tau}-D_{\alpha})\psi\|_{L^2} = \| (e^{iB_{\tau}} - e^{i B_0} ) \psi \|_{L^2}
\underset{\tau \to 0}{\longrightarrow} 0.$$
Finally, Lemma \ref{lem:reachable-operators} proves that $D_{\alpha}$ is $L^2$-STAR.

\medskip

\noindent \emph{3.} Let $\sigma >0$. For every $t>0$, the operator
$$L_{t} := D_{t^{1/2}}e^{i\sigma t(\Delta-V)}D_{t^{-1/2}} $$
is $L^2$-approximately reachable in time $(\sigma t)^+$, by Item \emph{2} and Lemma \ref{lem:reachable-operators}. 
Moreover, a rescaling argument proves that
\begin{equation}\label{eq:time-rescaling}
L_{t} =  
\exp \left( i\sigma \left(\Delta - tV(t^{1/2}\cdot) \right) \right) .
\end{equation}
Alternatively, one can also apply Proposition \ref{lem:conjugation}, which gives that 
$$L_t=\exp(D_{t^{1/2}}i\sigma t (\Delta-V)D_{t^{-1/2}}).$$ 
Then, for every $\psi\in C^\infty_c(\R^d,\C)$, explicit dilations and differentiations give the action of the generator: $$D_{t^{1/2}}i\sigma t (\Delta-V)D_{t^{-1/2}}\psi=i\sigma(\Delta-tV(t^{1/2}x))\psi,$$
which proves \eqref{eq:time-rescaling}.
For every $t>0$, the operator $\sigma \left(\Delta - tV(t^{1/2}\cdot) \right)$ is essentially self-adjoint on $C^{\infty}_c(\R^d,\C)$ by Proposition \ref{prop:self-adjointness} thus its closure $A_t$ is self-adjoint.
The operator $A_0:=\sigma \Delta$ is self-adjoint on $D(A_0):=H^2(\R^d,\C)$. $C^{\infty}_c(\R^d,\C)$ is a common core to $A_0$ and $A_{t}$. Let $\psi \in C^\infty_c(\R^d,\C)$ and $R>0$ such that $\text{supp}(\psi) \subset K:=B_{\R^d}(0,R)$. For $t>0$ small enough, i.e. $t< (\delta/R)^2$, using (\ref{Hyp:V}), we obtain
\begin{align*}
\|(A_{t}-A_0)\psi\|_{L^2} & = 
\| tV(t^{1/2}\cdot) \psi \|_{L^2}  =
t \left( \int_{K} |V(t^{1/2} x) \psi(x)|^2 dx \right)^{1/2}
\\ & \leq
t \left( \int_{K} c^2 |t^{1/2} x|^{2\gamma} |\psi(x)|^2 dx
\right)^{1/2} 
 \leq  c t^{1+\frac{\gamma}{2}} \left( \int_{K} |x|^{2\gamma} |\psi(x)|^2 dx
\right)^{1/2} \underset{t \to 0}{\longrightarrow} 0
\end{align*}
because the assumption on $\gamma$ ensures both the definition of the last integral and $1+\frac{\gamma}{2}>0$. This convergence, together with Proposition \ref{prop:trotter}, prove that, 
for every $\psi \in L^2(\R^d,\C)$,
$\|(L_{t} - e^{i\sigma \Delta})\psi  \|_{L^2} 
=\| (e^{i A_t} - e^{i A_0}) \psi \|_{L^2}
\underset{t \to 0}{\longrightarrow} 0$; and Lemma \ref{lem:reachable-operators} proves that $e^{i \sigma \Delta}$ is $L^2$-STAR.

\medskip

\noindent \emph{4.} For every $n \in \N^*$, the operator
$$L_n := \left( e^{i \frac{\sigma}{n} \Delta} e^{-i \frac{\sigma}{n}|x|^2} \right)^n $$
is $L^2$-STAR by Lemma \ref{lem:reachable-operators}, Item 1 and Item 3.
The operator $A:=\sigma \Delta$ is self-adjoint on $D(A):=H^2(\R^d,\C)$. The operator $B:=-\sigma |x|^2$ is self-adjoint on $D(B):=\{ \psi \in L^2(\R^d,\C) ; |x|^2 \psi \in L^2(\R^d,\C) \}$. The operator $A+B=\sigma(\Delta-|x|^2)$ is self-adjoint on $D(A) \cap D(B)$. Thus, by Proposition \ref{prop:trotter-kato}, for every $\psi \in L^2(\R^d,\C)$,
$\| (L_n-e^{i \sigma(\Delta-|x|^2)})\psi \|_{L^2} =
\| (e^{iA/n}e^{iB/n})^n \psi - e^{i(A+B)}\psi  \|_{L^2} \to 0 $ as $\tau \to 0$. By Lemma \ref{lem:reachable-operators}, this proves that $e^{i \sigma(\Delta-|x|^2)}$ is $L^2$-STAR.
\end{proof}

Now, we can justify Remark \ref{Rk:ex_W2_ab}. Let us consider the case $d=1$ and $W_2(x)=e^{ax^2+bx}$, with $a<0$ and $a,b$ algebraically independent. Then the Schrödinger equation \eqref{eq:small-time-oscillator_BIS}
is large-time approximately controllable (as proven in \cite[Proposition 6.4]{BCMS}) and the arguments above prove the same property in small time.

\section{A limiting example} \label{sec:Lim_Ex}

The goal of this section is to prove Theorem \ref{Thm:Lim_Ex}, as a corollary of a more precise statement, that is Theorem \ref{Prop:Lim_Ex}, presented in Subsection \ref{subsec:Descr}.

By \cite{fujiwara} (see Step 2 in the proof of Proposition \ref{Prop:WP}),
for every $\psi_0 \in L^2(\R^d,\C)$ and
$u=(u_0,u_1,\dots,u_d) \in L^{\infty}((0,T),\R^{d+1})$, 
there exists a unique solution  $\psi \in C^0([0,T],L^2(\R^d,\C))$ of
the Cauchy problem (\ref{eq:oscillator_BIS}).

\subsection{Description of the reachable set} \label{subsec:Descr}

\begin{definition} [Translation]
For $q \in \R^d$,
the translation $\tau_q$ is the linear isometry of $L^2(\R^d,\C)$ defined by
$(\tau_{q} \psi)(x)=\psi(x-q)$.
\end{definition} 

\begin{theorem}\label{Prop:Lim_Ex}
System (\ref{eq:oscillator_BIS}) satisfies the following properties.
\begin{enumerate}
\item For every $T>0$ and
$u=(u_0,u_1,\dots,u_d) \in L^{\infty}((0,T),\R^{d+1})$,
there exists 
$p,q \in \R^d$,
$\theta \in [0,2\pi)$,
$n \in \N^*$, 
$\alpha_1, \dots, \alpha_n \in \R$, 
$\beta_1, \dots \beta_n, \sigma_1, \dots, \sigma_n >0$ such that
\begin{equation} \label{Prod_gene_0}
\psi(T;u,.) =
e^{i( \langle p , x \rangle + \theta )} \tau_{q} 
e^{i \alpha_n |x|^2} D_{\beta_n} e^{i \sigma_n \Delta}
\dots
e^{i \alpha_1 |x|^2} D_{\beta_1} e^{i \sigma_1 \Delta}.
\end{equation}
\item  For every $\psi_0 \in \mathcal{S}$, the reachable set from $\psi_0$ 
\begin{equation} \label{Def:Reach}
\text{Reach}(\psi_0):=\{ e^{i\theta} \psi(T;u,\psi_0);
\theta \in [0,2\pi), T>0, u \in L^{\infty}(0,T) \}
\end{equation}
is not dense in 
$(\mathcal{S},\|.\|_{L^2})$. 
Thus, system (\ref{eq:oscillator_BIS}) is not $L^2$-approximately controllable.
\item For every $T>0$, $u \in L^{\infty}((0,T),\R^{d+1})$ and $\psi_0 \in L^2(\R^d,\C)$
\begin{equation} \label{5param}
\psi(T;u,\psi_0) \in \text{Adh}_{L^2} \{
e^{i( \langle p , x \rangle + \theta )} \tau_{q} e^{i \alpha |x|^2} D_{\beta} e^{i \sigma \Delta} \psi_0 ;
p,q \in \R^d, \alpha, \beta, \theta,\sigma \in \R
\}.
\end{equation}
\end{enumerate}
\end{theorem}

\begin{remark}
  The description of the state in (\ref{Prod_gene_0}) involves a product of exponentials of the generators of the Lie algebra generated by $i \Delta$, $i |x|^2$, $ix_1$, \dots, $ix_d$:
\begin{equation} \label{Lie}
\text{Lie}(i \Delta, i |x|^2, ix_1, \dots, ix_d ) = 
\text{Span}_{\R} \{i \Delta, i |x|^2,  ix_1, \dots, ix_d, \langle x , \nabla \rangle ,\partial_{x_1}, \dots, \partial_{x_d} , i \}.
\end{equation}
This representation formula evokes Sussmann's infinite product \cite{MR935387, P1}.
The Lie algebra (\ref{Lie}) has finite dimension, but it is not nilpotent, thus it is not expected that the representation formula of (\ref{Prod_gene_0}) hold with a fixed number $n$ (independent of $T$ and $u$). The last statement proves that one may take $n=1$ provided $\beta,\sigma \in \R$ (instead of $>0$) and the conclusion be approximate (instead of exact). Note that, in this case, $D_{\beta}$ with $\beta<0$ is not an exponential of the generators of (\ref{Lie}).
\end{remark}

The following result allows to focus on the subsystem (\ref{eq:trap}), see Appendix \ref{Appendix1} for a proof.

\begin{proposition}
Let $T>0$ and 
$u=(u_0, u_1, \dots, u_d) \in L^\infty((0,T),\R^{d+1})$.
Let  $p,q \in C^0([0,T],\R^d)$ be the solution of the following linear system and $\theta \in C^0([0,T],\R)$ be defined by
\begin{equation} \label{def:p,q,theta}
\begin{array}{ll}
\left\lbrace \begin{array}{l}
\dot{q}=p, \\
\dot{p}=-4 u_0 q - 2u, \\
(p,q)(0)=(0,0),
\end{array}\right.
& \quad
\theta(t):=\int_0^t\left( u_0(s)  |q(s)|^2 - \frac{1}{4} |p(s)|^2 \right) ds.
\end{array}
\end{equation}
Then, for every $\psi_0 \in \mathcal{S}$, the solution of (\ref{eq:oscillator_BIS}) is
\begin{equation} \label{reduc}
\psi(t)=e^{i \left( \frac{1}{2}  p(t) \cdot x  + \theta(t) \right)}  \tau_{q(t)} \xi(t)
\end{equation}
where
\begin{equation} \label{equation_de_xi}
\left\lbrace\begin{array}{l}
i \partial_t \xi = (-\Delta + u_0(t) |y|^2)\xi(t,y), \quad (t,y) \in (0,T)\times\R^d, \\
\xi(0,.)=\psi_0.
\end{array}\right.
\end{equation}
\end{proposition}

\subsection{Exactly reachable set}

Statement 1 of Theorem \ref{Prop:Lim_Ex} is a consequence of \eqref{reduc} and the following result.

\begin{proposition} \label{Prop:SL2}
System (\ref{eq:trap}) satisfies the following property:
for every $T>0$ and $u \in L^\infty(0,T)$
\begin{equation} \label{Prod_gene}
\psi(T;u,.) \in \{  
e^{i \alpha_n |x|^2} D_{\beta_n} e^{i \sigma_n \Delta}
\dots
e^{i \alpha_1 |x|^2} D_{\beta_1} e^{i \sigma_1 \Delta}  ; 
n \in \N^*, \alpha_j \in \R, \beta_j, \sigma_j>0  \}.
\end{equation}
In particular, for every $\psi_0 \in \mathcal{S}$, the exactly reachable set from $\psi_0$ (defined by (\ref{Def:Reach})) is not dense in $(\mathcal{S},\|.\|_{L^2})$.
\end{proposition}

\begin{remark}
The description of the state in (\ref{Prod_gene}) involves products of the exponentials of the generators of the Lie algebra generated by 
$i \Delta$ and $i|x|^2$:
$$\text{Lie}(i \Delta ,  i |x|^2 )= \text{Span}_{\R} \{
i \Delta,  
i|x|^2, 
\langle x. \nabla \rangle \}.$$
It is natural to wonder whether it is necessary to consider arbitrary long products of exponentials. Proposition \ref{Prop:SL2_BIS} below justifies that $3$ terms are sufficient, provided the parameters $\beta$ and $\sigma$ be real (instead of positive), and the conclusion is an approximation (instead of an equality).
\end{remark}

To prove Proposition \ref{Prop:SL2}, we will use the following 2 elementary results.

\begin{lemma} \label{Lem:EDO}
For every $\varepsilon \in \R$ and $\underline{u} \in L^1(0,1)$ such that $|\varepsilon| \|\underline{u}\|_{L^1}<1$, the maximal solution of
$$\left \lbrace \begin{array}{l}
\dot{\underline{a}}=\varepsilon \underline{a}^2 + \underline{u} \\
\underline{a}(0)=0
\end{array}\right.$$
is defined on $[0,1]$.    
\end{lemma}

\begin{proof}[Proof of Lemma \ref{Lem:EDO}:]
Let $\underline{u} \in L^1(0,1)$ and
$(b_k)_{k \in \N} \subset C^0([0,1],\R)$ be defined by induction by
$$b_0(\tau)= \int_0^{\tau} \underline{u}, \qquad 
b_k(\tau)=\int_0^{\tau} \left( \sum_{j=0}^{k-1} b_j(s) b_{k-1-j}(s) \right) ds.$$
By induction on $k \in \N$, one proves that, for every $\tau \in [0,1]$,
$|b_k(\tau)| \leq \tau^{k} \|\underline{u}\|_{L^1}$. 

\medskip

Let $\varepsilon \in \R$ be such that
$|\varepsilon| \|\underline{u}\|_{L^1}<1$.
The series 
$\underline{a} :=\sum_{k=0}^{\infty} \epsilon^{k} b_k$
converges uniformly, thus defines $\underline{a} \in C^0([0,1],\R)$. By definition of $(b_k)_{k\in\N}$,  it satisfies, for every $\tau \in [0,1]$
$$\underline{a}(\tau)=\int_0^{\tau} \left( \varepsilon a(s)^2 + \underline{u}(s) \right) ds.$$
\end{proof}

\begin{lemma} \label{Lem:Gauss}
Let $C:=(8/\pi)^{d/2}$. The set of normalized centered Gaussian functions
\begin{equation} \label{def:Gauss}
\mathcal{G}:= \{ x \in \R^d \mapsto e^{i \theta} C  a^{\frac{d}{4}} e^{-(a+ib)|x|^2}  ; \theta \in [0,2\pi), a>0, b \in \R \}
\end{equation}
is a strict closed subset of $(\mathcal{S},\|.\|_{L^2})$, stable by the operators $e^{i\alpha|x|^2}$ for $\alpha \in \R$, 
$D_{\beta}$ for $\beta>0$, $e^{i\sigma \Delta}$ for $\sigma \in \R$. 
\end{lemma}

\begin{proof} The strict inclusion and the stability properties are easy, thus we only prove the closedness. Let $(f_n)_{n\in\N}$ be a sequence of $\mathcal{G}$ and $f \in \mathcal{S}$ such that $\|f_n-f\|_{L^2} \to 0$ as $n \to \infty$. Let $\theta_n, a_n, b_n$ be the associated parameters. One may assume that $f_n \to f$ almost everywhere (otherwise take an extraction).

\medskip

\noindent \emph{Step 1: We prove that we can assume $\theta_n = 0$.} There exists $\theta \in [0,2\pi]$ such that (up to an extraction) $\theta_n \to \theta$. Then $e^{-i\theta_n} f_n \in \mathcal{G}$,
$e^{-i \theta} f \in \mathcal{S}$ and
$\| e^{-i \theta_n} f_n - e^{-i\theta} f\|_{L^2} \to 0$.

\medskip

\noindent \emph{Step 2: We prove that $(a_n)_{n\in\N}$ is bounded.} Reasoning by contradiction, we assume there exists an extraction $\varphi$ such that $a_{\varphi(n)} \rightarrow + \infty$. Then, for every $x \in \R^d\setminus \{0\}$,
$|f_{\varphi(n)}(x)|= C a_{\varphi(n)}^{d/4} e^{-a_{\varphi(n)}|x|^2} \to 0$. By uniqueness of the a.e. pointwise limit, $|f|=0$, which contradicts $f \in \mathcal{S}$.

\medskip

\noindent \emph{Step 3: We prove there exists 
$\overline{\Theta}:(0,1) \rightarrow \R$ such that
$\|e^{ib_n t} - e^{i \overline{\Theta}(t)} \|_{L^2(0,1)} \to 0$.}
By Step 1, one may assume that $a_n \to a$ where $a \in [0,\infty)$ (otherwise, take an extraction). Then 
$|f_{\varphi(n)}(x)|=C a_{\varphi(n)}^{d/4} e^{-a_{\varphi(n)}|x|^2} \to
C a^{d/4} e^{-a|x|^2}$. By uniqueness of the a.e. pointwise limit, $|f(x)|=C a^{d/4} e^{-a|x|^2}$ thus $a>0$.
Let $\Theta:\R^d \rightarrow \R$ such that 
$f(x)=C a^{d/4} e^{-a|x|^2} e^{i \Theta(x)}$. Then $e^{ib_n|x|^2} \to e^{i \Theta(x)}$ a.e., thus the dominated convergence theorem ends Step 3 with, for instance $\overline{\Theta}: t \in (0,1) \mapsto \Theta( \sqrt{t} e_1)$.

\medskip

\noindent  \emph{Step 4: We prove that $(b_n)_{n\in\N}$ is bounded.} Reasoning by contradiction, we assume there exists an extraction $\varphi$ such that $|b_{\varphi(n)}| \rightarrow + \infty$. By the Riemann-Lebesgue Lemma,
$e^{i b_{\varphi(n)} t}  \rightharpoonup 0$ in $L^2(0,1)$.
By Step 3 and the uniqueness of the weak $L^2(0,1)$-limit, we obtain $e^{i \overline{\Theta}(t)}=0$ for a.e. $t \in (0,1)$, which is a contradiction.   

\medskip

In conclusion, one may assume $b_n \to b$ where $b \in \R$ (otherwise take an extraction). Then, by uniqueness of the a.e. pointwise limit, $f(x)=C a^{d/4} e^{-(a+ib)|x|^2}$, i.e. $f \in \mathcal{G}$.
\end{proof}

\begin{proof}[Proof of Proposition \ref{Prop:SL2}:]  

\noindent \emph{Step 1: We prove that, if 
$T>0$ and
$u\in L^{\infty}((0,T),\R)$ satisfy
$4 T \|u\|_{L^1(0,T)} < 1$, then, for every $t \in [0,T]$,
\begin{equation} \label{psi(t)_explicit}
\psi(t;u,.)= e^{ia(t)|x|^2} D_{\frac{1}{b(t)}} e^{i \zeta(t) \Delta} 
\end{equation}
where $a$ is the solution of 
\begin{equation}  \label{eq:a}
\left\lbrace \begin{array}{l}
\dot{a}(t) = - 4 a(t)^2  - u(t), \\
a(0)=0.
\end{array}\right.
\end{equation}
and $b, \zeta$ are given by (\ref{def:b_zeta}).}  
It is sufficient to prove that the maximal solution of the nonlinear ODE (\ref{eq:a}) is defined on $[0,T]$, because then, standard computations prove that the right hand side of (\ref{psi(t)_explicit}) solves (\ref{eq:trap}).

Let $\varepsilon:=-4T$, 
$\underline{u}: s \in [0,1] \mapsto -T u(Ts)$ and 
$\underline{a}:[0,1] \rightarrow \R$ 
given by Lemma \ref{Lem:EDO}. Then 
$a:t \in [0,T] \mapsto \underline{a}(t/T)$ 
solves (\ref{eq:a}) on $[0,T]$.

\medskip

\noindent \emph{Step 2: We prove (\ref{Prod_gene}).} Let $T>0$ and $u \in L^{\infty}((0,T),\R)$. There exists $n \in \N^*$ and a subdivision $T_0=0<T_1<\dots<T_n=T$ such that 
$4(T_j-T_{j-1}) \|u\|_{L^1(T_{j-1},T_j)} < 1$ for $j=1,\dots,n$. We apply Step 1 on each intervals $(T_{j-1},T_j)$.

\medskip

\noindent \emph{Step 3: Let $\psi_0 \in \mathcal{S}$. We prove that the exactly reachable set from $\psi_0$ is not dense in $(\mathcal{S},\|.\|_{L^2})$.}

\medskip

\noindent \emph{Case 3.1: $\psi_0 \in \mathcal{G}$.} By (\ref{Prod_gene}) and Lemma \ref{Lem:Gauss},
$\text{Reach}(\psi_0) \subset \mathcal{G}$, 
thus $\text{Reach}(\psi_0)$ is not dense in $(\mathcal{S},\|.\|_{L^2})$.

\medskip

\noindent \emph{Case 3.2: $\psi_0 \notin \mathcal{G}$.} By Lemma \ref{Lem:Gauss}, $d:=\text{dist}_{L^2}( \psi_0, \mathcal{G})>0$ because $\mathcal{G}$ is closed. Let 
 $\psi_f \in \text{Reach}(\psi_0)$ and $g \in \mathcal{G}$.
By (\ref{Prod_gene}), 
$\psi_f=
e^{i \alpha_n |x|^2} D_{\beta_n} e^{i \sigma_n \Delta}
\dots
e^{i \alpha_1 |x|^2} D_{\beta_1} e^{i \sigma_1 \Delta} \psi_0$ 
for some
$n \in \N^*$, $\alpha_j \in \R$, $\beta_j, \sigma_j>0$ and
\begin{align*}
\| \psi_f - g \|_{L^2} 
&  =
\|e^{i \alpha_n |x|^2} D_{\beta_n} e^{i \sigma_n \Delta}
\dots
e^{i \alpha_1 |x|^2} D_{\beta_1} e^{i \sigma_1 \Delta} \psi_0 - g  \|_{L^2}
\\  & =
\| \psi_0 - 
e^{-i \sigma_1 \Delta} D_{\beta_1^{-1}}e^{- i \alpha_1 |x|^2}
\dots 
e^{-i \sigma_n \Delta} D_{\beta_n^{-1}} e^{-i \alpha_n |x|^2}g 
\|_{L^2}  \geq d >0
\end{align*}
because all the operators involved are isometries of $L^2(\R^d,\C)$ and they preserve $\mathcal{G}$. Thus $\text{Reach}(\psi_0)$ is not dense in $(\mathcal{S},\|.\|_{L^2})$.
\end{proof}

\subsection{Absence of approximate controllability}

The goal of this section is to prove the statement 2 of Theorem \ref{Prop:Lim_Ex}. We will use the following result.

\begin{lemma} \label{Lem:Gauss_bis}
Let $\mathcal{G}$ be defined by (\ref{def:Gauss}). The set 
$$\mathcal{G}':= \{ x \in \R^d \mapsto \tau_q e^{i \langle p, x \rangle} g(x); p,q \in \R^d, g \in \mathcal{G} \}$$
is a strict closed subset of $(\mathcal{S},\|.\|_{L^2})$, stable by the operators 
$\tau_q$ for $q\in\R^d,$
$e^{i\langle p , x \rangle}$ for $p \in\R^d$,
$e^{ia|x|^2}$ for $a \in \R$, 
$D_{\beta}$ for $\beta>0$, 
$e^{i\sigma \Delta}$ for $\sigma \in \R$. 
\end{lemma}

\begin{proof}
The strict inclusion and the stability properties are simple, thus we only prove the closedness. Let $(f_n)_{n\in\N}$ be a sequence of $\mathcal{G}'$ and $f \in \mathcal{S}$ such that $\|f_n-f\|_{L^2} \to 0$ as $n \to \infty$. Let $p_n, q_n, g_n$ be the associated parameters. One may assume that $f_n \to f$ a.e. (otherwise take an extraction).

\medskip

\noindent \emph{Step 1: We prove that $(q_n)_{n\in\N}$ is bounded.}
Reasoning by contradiction, we assume there exists an extraction $\varphi$ such that $|q_{\varphi(n)}| \to \infty$. We have 
$$
\|\, |g_n| - \tau_{-q_n} |f|\, \|_{L^2}
= \|\, \tau_{q_n} |g_n| - |f|\, \|_{L^2}
= \|\, |f_n| - |f|\, \|_{L^2} 
\leq \|f_n-f\|_{L^2} 
\underset{n \to \infty}{\longrightarrow} 0
$$
and
$\tau_{-q_{\varphi(n)}} |f| \rightharpoonup 0$ weakly in $L^2(\R^d)$,
thus $|g_{\varphi(n)}| \rightharpoonup 0$ weakly in $L^2(\R)$. By definition of $\mathcal{G}$, there exists $a_n>0$ such that
$|g_n(x)|=C a_n^{d/4} e^{-a_n|x|^2}$. Thus (up to an extraction), either  $a_{\varphi(n)} \to \infty$ or $a_{\varphi(n)} \to 0$. In any case, $|f_{\varphi(n)}(x)|=C a_{\varphi(n)}^{d/4} e^{-a_{\varphi(n)}|x-q_n|^2} \to 0$ a.e., which is a contradiction.

\medskip

\noindent \emph{Step 2: We prove that $(p_n)_{n\in\N}$ is bounded.} By Step 1, one may assume that $q_n \to q$ for some $q \in \R^d$ (otherwise take an extraction). Then 
$\tau_{-q_n} f_n \in \mathcal{G'}$,
$\tau_{-q} f \in \mathcal{S}$ and
$\| \tau_{-q_n} f_n - \tau_{-q} f \|_{L^2}$ thus, one may assume that $q_n=0$, i.e. $f_n(x)=e^{i \langle p_n , x \rangle} g_n(x)$.
Then $\hat{f}_n=\tau_{p_n} \hat{g}_n$ and
$\| \hat{f}_n - \hat{f} \|_{L^2} \to 0$. Thus, the argument of Step 1 proves that $(p_n)_{n\in\N}$ is bounded.

\medskip

\noindent \emph{Step 3: We prove that $f \in \mathcal{G}$.}
By Step 2, one may assume that $p_n \to p$ for some $p \in \R^d$. Then
$e^{-i\langle p_n,x\rangle} f_n \in \mathcal{G}'$,
$e^{-i\langle p , x \rangle} f \in \mathcal{S}$ and
$\|e^{-i\langle p_n,x\rangle} f_n-e^{-i\langle p , x \rangle} f\|_{L^2} \to 0$ thus we may assume that $p_n=0$.i.e. $f_n=g_n$.
By Lemma \ref{Lem:Gauss}, $f \in \mathcal{G}$, which ends the proof.
\end{proof}

\begin{proof}[Proof of Theorem \ref{Prop:Lim_Ex}, Statement 2:]
It is completely analogous to the step 3 in the proof of Proposition \ref{Prop:SL2} (we replace \eqref{Prod_gene} with \eqref{Prod_gene_0} and Lemma \ref{Lem:Gauss} with Lemma \ref{Lem:Gauss_bis}).




\end{proof}

\subsection{Approximately reachable set}

Statement 3 of Theorem \ref{Prop:Lim_Ex} is a consequence of (\ref{reduc}) and the following result.

\begin{proposition}\label{Prop:SL2_BIS}
System (\ref{eq:trap}) satisfies the following property:
for every $T>0$, $u \in L^\infty(0,T)$ and $\psi_0 \in L^2(\R^d,\C)$,
\begin{equation} \label{3param}
\psi(T;u,\psi_0) \in \text{Adh}_{L^2} \{  e^{i \alpha |x|^2} D_{\beta} e^{i \sigma \Delta} \psi_0 ; \alpha,\beta,\sigma \in \R \}. 
\end{equation}  
\end{proposition}

To prove Proposition \ref{Prop:SL2_BIS}, we will use the following commutation argument.

\begin{lemma} \label{Lem:com}
    Let $s \in \R^*$, $a \in \R$ such that $4 a s \neq -1$. Then
    $$e^{i s \Delta} e^{i a |.|^2} = e^{i \frac{a}{1+4as}|.|^2} D_{\frac{1}{1+4as}} e^{i\frac{s}{1+4as} \Delta}. $$
\end{lemma}

\begin{proof}[Proof of Lemma \ref{Lem:com}:] First, we recall that
\begin{equation} \label{exp(isDelta)}
e^{i s \Delta} = \frac{1}{(2i\pi)^{d/2}} e^{i \frac{|.|^2}{4s}} D_{\frac{1}{2s}} \mathcal{F} e^{i \frac{|.|^2}{4s}} 
\end{equation}
where $\mathcal{F}$ is the Fourier transform, with the following normalization
$$\forall f \in L^1(\R^d), \quad \mathcal{F}(f)(\xi)=\int_{\R^d} f(x) e^{-ix.\xi} dx. $$
The formula (\ref{exp(isDelta)}) is derived from the representation via a convolution product with the fundamental solution (see, e.g., \cite[eq. (7.43)]{Teschl-MMQM-2014}).
Let $t=\frac{s}{1+4as}$. We deduce from (\ref{exp(isDelta)}) that
$$
e^{i s \Delta} e^{i a |.|^2} 
= \frac{1}{(2\pi)^{d/2}}
e^{i \frac{|.|^2}{4s}} D_{\frac{1}{2s}} \mathcal{F} e^{i \frac{|.|^2}{4t}} 
=
e^{i \frac{|.|^2}{4s}} D_{\frac{t}{s}} 
e^{-i \frac{|.|^2}{4t}}
e^{i t \Delta}
=
e^{i \frac{|.|^2}{4s}} 
e^{-i \frac{|.|^2}{4t} \frac{t^2}{s^2} }
D_{\frac{t}{s}} 
e^{i t \Delta}
$$
which gives the conclusion.
\end{proof}

\begin{proof}[Proof of Proposition \ref{Prop:SL2_BIS}]
Let $\psi_0 \in L^2(\R^d,\C)$. We prove by induction on $n \in \N^*$ that, for every 
$T>0$ and $u\in L^{\infty}((0,T),\R)$ satisfying
$4 T \|u\|_{L^1(0,T)} < n$, then
$$\psi(T;u,\psi_0)\in
\text{Adh}_{L^2} \{  e^{i \alpha |x|^2} D_{\beta} e^{i \sigma \Delta} \psi_0 ; \alpha,\beta,\sigma \in \R \}.$$

\medskip

\emph{Initialization:} If $n=1$ then the Step 1 in the proof of Proposition \ref{Prop:SL2} gives the conclusion with
$(\alpha,\beta,\sigma)=(a,1/b,\zeta)(T)$.

\medskip

\emph{Heredity:} Let $n \geq 2$. We assume the property proved up to $(n-1)$. Let $T>0$ and $u\in L^{\infty}((0,T),\R)$ satisfying
$4 T \|u\|_{L^1(0,T)} < n$. Let $T_1:= \frac{n-1}{n} T$. Then
$$
4T_1\|u\|_{L^1(0,T_1)} \leq 4 T \frac{n-1}{n} \|u\|_{L^1(0,T)}
< n-1 
\quad 
\text{ and } 
\quad 
4(T-T_1)\|u\|_{L^1(T_1,T)} \leq \frac{4 T}{n} \|u\|_{L^1(0,T)} < 1.
$$
thus we can apply the induction assumption on $(0,T_1)$ and 
Step 1 in the proof of Proposition \ref{Prop:SL2} on $(T_1,T)$.

Let $\epsilon>0$. There exists 
$\alpha_1, \beta_1, \sigma_1, \alpha_2 \in \R$, 
$\beta_2,\sigma_2>0$ such that
\begin{equation} \label{egalité_T1/T}
\psi(T;u,\psi_0)=
e^{i\alpha_2 |x|^2} D_{\beta_2} e^{i \sigma_2 \Delta}
\psi(T_1;u,\psi_0)    
\end{equation}
and
\begin{equation} \label{Estim_T1}
\| \psi(T_1;u,\psi_0) - 
e^{i\alpha_1|x|^2} D_{\beta_1} e^{i \sigma_1 \Delta}\psi_0
\|_{L^2} < \frac{\epsilon}{2}.
\end{equation}

\emph{First case: $4 \sigma_2 \alpha_1 \neq -1$.} We deduce from (\ref{egalité_T1/T}) and (\ref{Estim_T1}) that
$$\| \psi(T;u,\psi_0) - 
e^{i\alpha_2 |x|^2} D_{\beta_2} e^{i \sigma_2 \Delta}
e^{i\alpha_1|x|^2} D_{\beta_1} e^{i \sigma_1 \Delta} \psi_0
\|_{L^2} < \frac{\epsilon}{2}.$$
By Lemma \ref{Lem:com}, there exists $\alpha_3, \beta_3 \in \R$ and $\sigma_3 \in \R^*$ such that
$$ e^{i \sigma_2 \Delta}  
e^{i \alpha_1 |.|^2}  
=e^{i \alpha_3 |.|^2} D_{\beta_3} 
e^{i \sigma_3 \Delta}.$$
Moreover, the invariant rescaling of the Schrödinger equation proves
$$e^{i \sigma_3 \Delta} D_{\beta_1} = 
D_{\beta_1} e^{i \frac{\sigma_3}{\beta_1^2} \Delta},$$ 
thus
$$
e^{i\alpha_2 |x|^2} D_{\beta_2} 
e^{i \sigma_2 \Delta} e^{i\alpha_1|x|^2} 
D_{\beta_1} e^{i \sigma_1 \Delta} 
=
e^{i\alpha_2 |x|^2} D_{\beta_2} 
e^{i \alpha_3 |.|^2} D_{\beta_1 \beta_3} e^{i \sigma_4 \Delta} 
\\ =
e^{i \alpha_4 |.|^2}
D_{\beta_4} 
e^{i \sigma_4 \Delta}
$$
where
$\sigma_4=\frac{\sigma_3}{\beta_1^2} + \sigma_1 $,
$\beta_4=\beta_2 \beta_3 \beta_1 $
and
$\alpha_4= \alpha_2 + \alpha_3\beta_2^2$.

\medskip

\emph{Second case: $4 \sigma_2 \alpha_1 = -1$.} There exists $\delta>0$ such that
$$\| \psi(T;u,\psi_0) - e^{i \delta \Delta} \psi(T;u,\psi_0) 
\|_{L^2} < \frac{\epsilon}{2} \quad 
\text{ and } \quad 
(\delta+T-T_1) \|u\|_{L^1(0,T)}<1.$$
We extend the control $u$ by zero on the interval $(T,T+\delta)$ and we apply Step 1 in the proof of Proposition \ref{Prop:SL2} on the interval $(T_1,T+\delta)$:
there exists $\alpha_2', \beta_2' \in \R$ and  $\sigma_2'>0$ such that 
$$e^{i \delta \Delta} \psi(T;u,\psi_0) = \psi(T+\delta; u ,\psi_0) = 
e^{i\alpha_2' |x|^2} D_{\beta_2'} e^{i \sigma_2' \Delta}
\psi(T_1;u,\psi_0).$$
The explicit expression of $\sigma_2'$ given in Step 1 of the proof of Proposition \ref{Prop:SL2} proves that $\sigma_2' >\sigma_2$ (because $t \mapsto \zeta(t)$ is increasing).
In particular, $4 \alpha_1 \sigma_2' \neq -1$ and the arguments of the previous case provide $\alpha_4, \beta_4, \sigma_4 \in \R$ such that
$$\|e^{i \delta \Delta} \psi(T;u,\psi_0) - 
\phi  \|_{L^2} < \frac{\epsilon}{2}
\quad \text{ where } \quad
\phi := e^{i \alpha_4 |.|^2}
D_{\beta_4} 
e^{i \sigma_4 \Delta} \psi_0.$$
Finally, by the triangle inequality,
$$ \| \psi(T;u,\psi_0) - \phi \|_{L^2}
 \leq 
\| \psi(T;u,\psi_0) - e^{i \delta \Delta} \psi(T;u,\psi_0) 
\|_{L^2} +\|e^{i \delta \Delta} \psi(T;u,\psi_0) - 
\phi  \|_{L^2}
< \epsilon.$$
\end{proof}

\section{The toy model at higher regularity}
\label{sec:TM_reg}

In this section, we prove that small-time approximate controllability of system (\ref{eq:small-time-oscillator}) remains true for stronger topologies.
To this end, we introduce the positive self adjoint operator
\begin{equation} \label{OH}
D(A)=  \{ f \in L^2(\R^d,\C); A f \in L^2(\R^d,\C)\},\, \quad
A = -  \Delta + |x|^2
\end{equation}
and, for $s \in \N$, the normed vector-space
\begin{equation} \label{s-norm}
D(A^s)=\{ f \in L^2(\R^d,\C);  A^s f \in L^2(\R^d,\C) \},
\qquad
\|f\|_{D(A^s)} = \| A^s f \|_{L^2}.
\end{equation}

\begin{theorem}\label{thm:small-time-global_smooth}
Let $s'>s \in \N$.
There exists a dense subset $\mathcal{D}$ of $W^{2s',\infty}(\R^d,\R)$ such that, for every $W_2 \in \mathcal{D}$,
the system (\ref{eq:small-time-oscillator}) is small-time $D(A^s)$-approximately controllable. 
\end{theorem}

\subsection{Well posedness}

\begin{proposition} \label{Prop:WP_D(As)}
Let 
$s \in \N^*$, 
$W_2 \in W^{2s,\infty}(\R^d,\R)$,
$T>0$, 
$(u_1, u_2) \in \mathcal{U}(0,T)$
(see Definition \ref{Def:Adm_cont})
and 
$\psi_0 \in D(A^s)$. 
The Cauchy problem (\ref{eq:small-time-oscillator}) has a solution 
$\psi \in C^0([0,T],D(A^s))$. Moreover, there exists $C=C(s,W_2,T,u)>0$ such that, for every $t \in [0,T]$,
$\| \psi(t;u,.) \|_{\mathcal{L}(D(A^s))} \leq e^{Ct}$.
\end{proposition}

\begin{proof} First, we recall, for every $s\in\N$, the equivalence between the following  norms
\begin{equation} \label{norm_eq}
\|f\|_{D(A^{s})} \, \sim \,
\sum_{\substack{
\alpha,\beta \in \N^d, 
\\ 
|\alpha|+|\beta| \leq 2s
}} \| x^{\alpha} \partial_x^{\beta} f \|_{L^2}.
\end{equation}

\noindent \emph{Step 1: We prove that for every $s\in\N^*$ and
$W_2 \in W^{2s,\infty}(\R^d,\R)$, the commutator $[A^s,|x|^2+W_2]$ maps continuously $D(A^s) $ into $L^2(\R^d,\C)$.}
The explicit expression $[A,|x|^2+W_2]=-2 \langle 2x+\nabla W_2,\nabla \rangle - (2d+\Delta W_2) $  and the equivalence (\ref{norm_eq}) prove that $[A,|x|^2+W_2]$ maps continuously $D(A^{k+1})$ into $D(A^k)$, for every $k \leq (s-1)$.
Thus Step 1 is a consequence  of the formula
$$[A^s,|x|^2+W_2] = \sum_{k=0}^{s-1} A^k [A,|x|^2+W_2] A^{s-1-k}.$$

\noindent \emph{Step 2: On a time interval $[t_1,t_2]$ on which $u_1$ and $u_2$ are constant.} Let $v_1:=u_1-1$. Then, for sufficiently regular initial conditions, using Step 1,
$$\begin{array}{ll}
\frac{1}{2} \frac{d}{dt} \| A^s \psi(t) \|_{L^2}^2
& 
= \Re \langle A^s \psi(t) , A^s (-i) (A+v_1|x|^2+u_2 W_2) \psi(t) \rangle \\
& = - \Im \langle A^s \psi(t) , [A^s  , v_1|x|^2+u_2 W_2 ] \psi(t) \rangle
\leq C \| A^s \psi(t) \|_{L^2}^2
\end{array} $$
where $C=C(s,u,W_2)>0$, thus $\|\psi(t)\|_{D(A^s)} \leq \|\psi(0)\|_{D(A^s)} e^{C(t-t_1)}$.

\medskip

\noindent \emph{Step 3: On a time interval $[t_1,t_2]$ on which $u_2=0$.} For sufficiently regular initial conditions
$$\begin{array}{ll}
\frac{1}{2} \frac{d}{dt} \| A^s \psi(t) \|_{L^2}^2
& = -  v_1(t) \Im \langle A^s \psi(t) , [A^s  , |x|^2 ] \psi(t) \rangle
\leq C  \| A^s \psi(t) \|_{L^2}^2
\end{array} $$
where $C=C(s,u_1)>0$, thus $\|\psi(t)\|_{D(A^s)} \leq \|\psi(0)\|_{D(A^s)} e^{C(t-t_1)}$. 

\medskip

The conclusion of Proposition \ref{Prop:WP_D(As)} is obtained by concatenating these propagators.

\end{proof}

\subsection{Large-time approximate controllability}\label{A:Boscain&al}

The large time $L^2$-approximate controllability of system (\ref{eq:small-time-oscillator}) is known to hold for a dense subset of potentials $W_2 \in L^\infty(\R^d,\R)$ (see Proposition \ref{Prop:Boscain&al}). This result can be extended to stronger norms.

\begin{proposition} \label{Prop:Boscain&al_smooth}
Let $s \in \N$ and $s' \geq 2s$. The system (\ref{eq:small-time-oscillator}) is large time $D(A^s)$-approximately controllable, generically with respect to $W_2 \in W^{s',\infty}(\R^d,\R)$. More precisely, there exists a dense subset $\mathcal{D}$ of $W^{s',\infty}(\R^d,\R)$ such that, for every $W_2 \in \mathcal{D}$ and $\psi_0 \in D(A^s) \cap \mathcal{S}$, the set
$$
\{e^{i \sigma_k( \Delta- |x|^2+\alpha_k W_2) } \dots  
e^{i \sigma_1(\Delta-|x|^2+\alpha_1 W_2 )}\psi_0\, ;\, k \in \N^*, \sigma_1,\dots,\sigma_k\geq 0, \alpha_1,\dots,\alpha_k \in\R \}$$
is dense in $D(A^s) \cap \mathcal{S}$.
\end{proposition}

To prove Proposition \ref{Prop:Boscain&al_smooth} we adapt the strategy of 
\cite[Theorem 2.6]{BCCS} and \cite[Proposition 4.6]{MS-generic}
to a different functional framework.

\begin{proof}
By \cite[Corollary 2.13]{BCS} and \cite[Proposition 2.10]{CS}, the existence of a non-resonant chain of connectedness implies the $D(A^s)$-approximate controllability result for \eqref{eq:small-time-oscillator}, in conclusion of Proposition \ref{Prop:Boscain&al_smooth}.
Thus, to get Proposition \ref{Prop:Boscain&al_smooth}, it suffices to prove that, for every $s \in \N^*$, the set $\mathcal{D}$ of $W_2\in W^{s,\infty}(\R^d,\R)$ such that system \eqref{eq:small-time-oscillator} admits a non-resonant chain of connectedness is dense in $W^{s,\infty}(\R^d,\R)$. 
  
\medskip   
  
\noindent \emph{Step 1:  Given $k\in\N$ and $q\in\mathbb{Q}^k\setminus\{0\}$, we prove the openness and density in $W^{s,\infty}(\R^d,\R)$ of the set $\mathcal{O}_q:=\widetilde{\mathcal{O}_q}\cap W^{s,\infty}(\R^d,\R)$, where
$$\widetilde{\mathcal{O}_q}:=\{W\in L^{\infty}(\R^d,\R); \sum_{j=1}^k q_j\lambda_j(-\Delta+|x|^2+W)\neq 0\}$$
and $\lambda_j(-\Delta+|x|^2+W)$ denotes the $j$-th eigenvalue of $-\Delta+|x|^2+W$.} 
This set is open in $W^{s,\infty}(\R^d,\R)$ because the eigenvalues $\lambda_j(-\Delta+|x|^2+W)$ are continuous w.r.t. variations of $W$ in $L^\infty$ and in particular in $W^{s,\infty}$. 
To prove its density in $W^{s,\infty}(\R^d,\R)$, we consider $\widetilde{W}\in \widetilde{\mathcal{O}_q}$, whose existence is guaranteed by \cite[Proposition 4.6]{MS-generic} (indeed, notice that the notion of "effective quadrupole" used in the statement of \cite[Proposition 4.6]{MS-generic} in our case exactly means that there exists a dense set of control potentials $\widetilde{W}\in L^\infty$ satisfying the spectral conditions given in the definitions of $\widetilde{\mathcal{O}_q}$ and $\widetilde{\mathcal{Q}_n}$ below). Now consider a sequence $(W_n)_{n\in\N}\subset W^{s,\infty}(\R^d,\R)$ such that $\|W_n-\widetilde{W}\|_{L^\infty} \to 0$ as $n \to \infty$ and such that the spectrum of $-\Delta+|x|^2+W_n$ is simple for all $n\in\N$; the existence of such a sequence is guaranteed by \cite{albert-75}. By continuity of the spectrum as a function of $W\in L^\infty$, we have that
$$\sum_{j=1}^k q_j\lambda_j(-\Delta+|x|^2+W_n)
\underset{n \to \infty}{\longrightarrow}
\sum_{j=1}^k q_j\lambda_j(-\Delta+|x|^2+\widetilde{W})\neq 0,$$
hence $W_n\in \mathcal{O}_q$ for $n=\overline{n}$ large enough. We now consider an analytic path $[0,1]\ni\mu\mapsto W(\mu)\in W^{s,\infty}(\R^d,\R)$ such that $W(0)=0, W(1)=W_{\overline{n}}-W$, and the spectrum of $-\Delta+|x|^2+W+W(\mu)$ is simple for all $\mu\in[0,1]$; the existence of such a path can be proved as in \cite[Proposition 2.12]{MS-generic}. Since the map $\mu\mapsto\sum_{j=1}^k q_j\lambda_j(-\Delta+|x|^2+W+W(\mu))$ is analytic and different from zero at $\mu=1$, it is different from zero for a.e. $\mu\in[0,1]$; in particular, for $\mu$ close to zero, we find $W+W(\mu)$ close to $W$ in $W^{s,\infty}$ belonging to $\mathcal{O}_q$.

\medskip

\noindent \emph{Step 2: Given $n\in\N$, we prove the openness and density in $W^{s,\infty}(\R^d,\R)$ of the set $\mathcal{Q}_n:=\widetilde{\mathcal{Q}_n}\cap W^{s,\infty}(\R^d,\R)$ where
\begin{align*}
\widetilde{\mathcal{Q}_n}:=\{&W\in L^{\infty}(\R^d,\R); \forall j,k\in\{1,\dots,n\}\exists r_1,\dots,r_l\in\N, r_1=j,r_l=k, \text{ with }\\
&\lambda_{r_i}(-\Delta+|x|^2+W) \text{ simple } \forall i=1,\dots,l, \text{ and }\\
&\int_{\R^d}W\phi_{r_i}(-\Delta+|x|^2+W)\phi_{r_{i+1}}(-\Delta+|x|^2+W)\neq 0\,\, \forall i=1,\dots,l-1\},
\end{align*}
where $\phi_j(-\Delta+|x|^2+W)$ denotes the $j$-th eigenfunction of $-\Delta+|x|^2+W$.} 
We argue exactly as in Step 1: we exploit the continuity of $W^{s,\infty}\ni W\mapsto |W|^{1/2}\phi_j(-\Delta+|x|^2+W)\in L^2$, and the existence of a reference $\widetilde{W}\in \widetilde{\mathcal{Q}_n}$ (whose existence is guaranteed again by \cite[Proposition 4.6]{MS-generic}).

\medskip

\noindent \emph{Step 3: Conclusion.}
The set of control operators $W_2\in W^{s,\infty}(\R^d,\R)$ such that system \eqref{eq:small-time-oscillator} admits a non-resonant chain of connectedness is given by 
$$\mathcal{D}=\cap_{n\in\N^*} \mathcal{Q}_n \cap_{k\in\N, q \in \mathbb{Q}^k\setminus \{0\}} \mathcal{O}_q.$$
By applying Baire's Theorem, this set is dense in $W^{s,\infty}(\R^d,\R)$.
\end{proof}

\subsection{Small time $D(A^s)$-approximate controllability}

Taking into account Proposition \ref{Prop:Boscain&al_smooth},
our strategy to prove Theorem \ref{thm:small-time-global_smooth} relies on the following result.

\begin{proposition} \label{Prop:Approx}
Let  $s \in \N$ and $W_2 \in W^{2(s+1),\infty}(\R^d,\R)$.
System (\ref{eq:small-time-oscillator}) satisfies the following property: for every
$\sigma \geq 0$, $\alpha \in \R$, the operator
$e^{i \sigma(\Delta-|x|^2+\alpha W_2)}$
is $D(A^s)$-STAR.
\end{proposition}

The proof of Proposition \ref{Prop:Approx} (as the one of Proposition \ref{Prop:Approx_BIS}) relies on the small-time exact reachability of $e^{i\sigma(\Delta-|x|^2)}$ (i.e. Proposition \ref{Prop_exact}) and the Trotter-Kato formula.
For the convergence in this formula to hold in the regular spaces $D(A^s)$, we use additional ingredients, explicited in the following statement.

\begin{proposition} \label{prop:trotter-kato_smooth}
Let $A,B, \mathcal{H}$ be as in Proposition \ref{prop:trotter-kato} and $X$ be a dense vector subspace of $\mathcal{H}$ equipped with a norm $\|.\|_{X}$. We assume there exists $C>0$ such that, for every $t \in [0,1]$, $e^{itA}, e^{itB}, e^{i(A+B)}$ are uniformly bounded operators on $X$ and
\begin{equation} \label{hyp_AB}
\| e^{tA} \|_{\mathcal{L}(X)},
\| e^{tB} \|_{\mathcal{L}(X)} \leq e^{C t}.
\end{equation}
Then, for every (strict) interpolation space $Y$ between $\mathcal{H}$ and $X$,
and for every $\psi_0 \in Y$,
$$ \left\| \left(e^{i\frac{B}{n}}e^{i\frac{A}{n}}\right)^n\psi_0 - e^{i(A+B)}\psi_0\right\|_{Y} \underset{n \rightarrow + \infty}{\longrightarrow} 0$$
\end{proposition}

\begin{proof}[Proof of Proposition \ref{prop:trotter-kato_smooth}]
There exists $\theta \in (0,1)$ such that, 
\begin{equation} \label{interp}
\forall \phi_0 \in X, \qquad 
\|\phi_0\|_{Y} \leq \|\phi_0\|_{\mathcal{H}}^{\theta} \|\phi_0\|_{X}^{1-\theta}.
\end{equation}

\medskip

\noindent \emph{Step 1: Convergence for $\psi_0 \in X$.} We deduce from (\ref{interp}), the triangular inequality and (\ref{hyp_AB}) that
$$\left\| \left(e^{i\frac{B}{n}}e^{i\frac{A}{n}}\right)^n\psi_0 - e^{i(A+B)}\psi_0\right\|_{Y}
\leq 
\left\| \left(e^{i\frac{B}{n}}e^{i\frac{A}{n}}\right)^n\psi_0 - e^{i(A+B)}\psi_0\right\|_{\mathcal{H}}^{\theta}
\left( e^C\|\psi_0\|_{X} +  \left\| e^{i(A+B)}\psi_0\right\|_{X} \right)^{1-\theta} $$
which gives the conclusion.

\medskip

\noindent \emph{Step 2: Convergence for $\psi_0 \in Y$.} Let $\psi_0 \in \mathcal{H}$, $\widetilde{\psi}_0 \in X$, $S_n:=\left(e^{i\frac{B}{n}}e^{i\frac{A}{n}}\right)^n$ and $S=e^{i(A+B)}$.
Then
$$
\|  (S_n-S) \psi_0  \|_{Y}
\leq 
\| (S_n-S) \widetilde{\psi}_0 \|_{Y}
 + 2 e^{C'}\| \widetilde{\psi}_0-\psi_0 \|_{Y}
$$
for some constant $ C'>0$; indeed, $S_n, S$ are bounded operators on $\mathcal{H}$ and $X$ thus also on $Y$. The density of $X$ in $Y$ gives the conclusion.
\end{proof}

\begin{proof}[Proof of Proposition \ref{Prop:Approx}:]
Let $\sigma > 0$, $\alpha \in \R$.
Let 
$\varepsilon \in (0,\sigma)$, 
$\sigma':=\sigma-\varepsilon$ and 
$n \in \N^*$.

\noindent \emph{Step 1: We prove that the operator $L_n$ is exactly reachable in time $\varepsilon^+$, where
$$L_n:=\left( e^{i\frac{\sigma'}{n}(\Delta-|x|^2)}
e^{i\frac{\varepsilon}{n}(\Delta-|x|^2+\frac{\alpha \sigma}{\varepsilon} W_2)}
\right)^n. $$}
By Proposition \ref{Prop_exact}, the operator 
$e^{i\frac{\sigma'}{n}(\Delta-|x|^2)}$ 
is small-time exactly reachable. The operator
$e^{i\frac{\varepsilon}{n}(\Delta-|x|^2+\frac{\alpha \sigma}{\varepsilon}W_2)}$ is exactly reachable in time $\varepsilon/n$ because associated with the constant control $u=(1,-\frac{\alpha \sigma}{\varepsilon})$. Therefore, Lemma \ref{lem:reachable-operators} ends Step 1.

\medskip

\noindent \emph{Step 2: We apply  Proposition \ref{prop:trotter-kato_smooth} with 
$\mathcal{H}=L^2(\R^d,\C)$,
$X=D(A^{s+1})$, 
$Y=D(A^s)$,
$A=\sigma'(\Delta-|x|^2)$ and
$B=\varepsilon (\Delta-|x|^2+\frac{\alpha \sigma}{\varepsilon}W_2)$.}
The operators $A$ and $B$ are essentially self-adjoint on $C^\infty_c(\R^d,\C)$ and so does their sum, by Proposition \ref{prop:self-adjointness}. The bounds (\ref{hyp_AB}) are proved in Proposition \ref{Prop:WP_D(As)}, because $W_2 \in W^{2(s+1),\infty}(\R^d,\R)$. Thus, for every $\psi \in D(A^s)$,
$\|(L_n-e^{i \sigma(\Delta-|x|^2+\alpha W_2)})\psi\|_{D(A^s)} \to 0$ as $n \to \infty$. By Lemma \ref{lem:reachable-operators}, this proves that the operator $e^{i \sigma(\Delta-|x|^2+\alpha W_2)}$ is $D(A^s)$-approximately reachable in time $\varepsilon^+$. This holds for any $\varepsilon \in (0,\sigma)$ thus, the operator $e^{i \sigma(\Delta-|x|^2+\alpha W_2)}$ is $D(A^s)$-STAR.      
\end{proof}

\appendix
\section{Appendix} \label{Appendix1}

In this section, we gather the proofs of the formulas (\ref{explicit}) and (\ref{reduc}).

\begin{proof}[Proof of formula (\ref{explicit})]
Let $\psi_0, u, \psi$ be such that (\ref{eq:trap}) holds and 
$\xi$ be defined by 
$$\psi(t,x)=\xi\left( \zeta(t) , \frac{x}{b(t)} \right) \frac{e^{i a(t) |x|^2}}{b(t)^{d/2}}$$
where $a,b,\zeta$ are defined by (\ref{eq:abxi}) (see also (\ref{def:b_zeta})).
To prove (\ref{explicit}), it suffices to prove that $\xi$ solves 
$$\left\lbrace \begin{array}{l}
i \partial_{\tau} \xi(\tau,y) = (-\Delta+|y|^2) \xi(\tau,y), \qquad (\tau,y)\in \R_+ \times\R^d, \\
\xi(0,y)=\psi_0(y).
\end{array}\right.$$
One may assume $\psi_0, u$ sufficiently regular so that $\psi$ and $\xi$ are $C^1$ w.r.t. the time-variable and $C^2$ w.r.t. the space variable. Then, the chain rule proves that, for every $(t,x) \in (0,T) \times \R^d$,
$$
i \partial_t \psi(t,x)=  
\left\{ 
\left( - \dot{a}(t)|x|^2 - i \frac{d \dot{b}(t)}{2 b(t)} \right)
\xi   
+
i \dot{\zeta}(t) \partial_{\tau} \xi 
- i \frac{\dot{b}(t)}{b(t)^2}  x  \cdot \nabla_y \xi  
\right\} 
\left( \zeta(t), \frac{x}{b(t)} \right)
\frac{e^{i a(t) |x|^2}}{b(t)^{d/2}},
$$
$$
\partial_{x_j} \psi(t,x) =
\left\{ 
\frac{1}{b(t)} \partial_{y_j} \xi 
+ 2 i a(t) x_j \xi
\right\} 
\left( \zeta(t), \frac{x}{b(t)} \right)
\frac{e^{i a(t) |x|^2}}{b(t)^{d/2}},
$$
$$
\partial_{x_j}^2 \psi(t,x) =
\left\{ 
\frac{1}{b(t)^2} \partial_{y_j}^2 \xi 
+ 4 i \frac{a(t)}{b(t)} x_j \partial_{y_j}\xi
+ \left( 2 i a(t) -4 a(t)^2 x_j^2 \right) \xi
\right\} 
\left( \zeta(t), \frac{x}{b(t)} \right)
\frac{e^{i a(t) |x|^2}}{b(t)^{d/2}},
$$
thus
$$\begin{aligned}
0  = & \left(i \partial_t + \Delta_x -u(t)|x|^2 \right)\psi(t,x)
\\  = &
\left\{ 
i \dot{\zeta}(t) \partial_{\tau} \xi 
+ \frac{1}{b(t)^2} \Delta_y \xi 
+ \left(   4 i \frac{a(t)}{b(t)} 
- i \frac{\dot{b}(t)}{b(t)^2} \right) x  \cdot \nabla_y \xi
+\left( - \dot{a}(t) - 4 a(t)^2 - u(t)  \right)|x|^2 \xi
\right.
\\ & \left.
+ \left( - i \frac{d \dot{b}(t)}{2 b(t)} - 2 i d  a(t)  \right) \xi
\right\} 
\left( \zeta(t), \frac{x}{b(t)} \right)
\frac{e^{i a(t) |x|^2}}{b(t)^{d/2}}.
\end{aligned}$$
Using (\ref{eq:abxi}), we conclude that, for every $(t,x) \in (0,T) \times \R^d$,
$$ 0 = \frac{1}{b(t)^2}
\left\{ i  \partial_{\tau} \xi 
+  \Delta_y \xi
- \left|\frac{x}{b(t)}\right|^2 \xi
\right\} 
\left( \zeta(t), \frac{x}{b(t)} \right)
$$
which gives the conclusion.
\end{proof}

\medskip

\begin{proof}[Proof of formula (\ref{reduc})]
Let $\psi_0, u, \psi, p,q,\theta$ be such that (\ref{eq:oscillator_BIS}) and (\ref{def:p,q,theta}) hold and $\xi$ be defined by
\begin{equation} 
\psi(t,x)=  \xi(t,x-q(t)) e^{i \left( \frac{1}{2} p(t) \cdot x  + \theta(t) \right)}
\end{equation}
To prove (\ref{reduc}), it suffices to prove that $\xi$ solves (\ref{equation_de_xi}). One may assume $\psi_0$ and $u$ regular enough so that $\psi$ and $\xi$ are $C^1$ w.r.t. the time-variable and $C^2$ w.r.t. the space-variable. Then, by the chain rule, for every $(t,x) \in (0,T)\times\R^d$,
$$
i \partial_t \psi(t,x)=\left\{ 
i \partial_t \xi
- i \dot{q}(t) \cdot \nabla_y \xi
- \left( \frac{1}{2}  \dot{p}(t) \cdot x  + \dot{\theta}(t) \right) \xi
\right\}(t,x-q(t)) e^{i \left( \frac{1}{2}  p(t) \cdot x  + \theta(t) \right)},
$$
$$
\partial_{x_j} \psi(t,x) =
\left\{
\partial_{y_j} \xi
+\frac{i}{2} p_j(t) \xi
\right\}(t,x-q(t)) e^{i \left( \frac{1}{2}  p(t) \cdot x  + \theta(t) \right)},
$$
$$
\partial_{x_j}^2 \psi(t,x) =
\left\{
\partial_{y_j}^2 \xi
+i p_j(t) \partial_{y_j} \xi
- \frac{p_j(t)^2}{4} \xi
\right\}(t,x-q(t)) e^{i \left( \frac{1}{2}  p(t) \cdot x  + \theta(t) \right)},
$$
thus
$$\begin{aligned}
0 & = \left(i \partial_t + \Delta - u_0(t)|x|^2 - u(t) \cdot x\right) \psi(t,x) e^{-i \left( \frac{1}{2}  p(t) \cdot x  + \theta(t) \right)}
\\ & = 
\Big\{ 
i \partial_t \xi
+ \Delta \xi
+ i( p(t)-\dot{q}(t) ) \cdot \nabla_y \xi
-  u_0(t) |x-q(t)|^2 \xi
 \\ & -  \left( 
 \left(
\frac{1}{2} \dot{p} + 2 u_0(t) q(t) + u(t) 
\right) \cdot x 
+ \left(
\dot{\theta}(t) + \frac{|p(t)|^2}{4} - u_0(t) |q(t)|^2 
\right)
 \right) \xi
\Big\}(t,x-q(t)).
\end{aligned}$$
Using (\ref{def:p,q,theta}), we obtain, for every $(t,x) \in \R_+ \times \R^d$
$$
0= \Big\{ 
i \partial_t \xi
+ \Delta \xi
-  u_0(t) |x-q(t)|^2 \xi
\Big\}(t,x-q(t)),
$$
which gives the conclusion.
\end{proof}

\textbf{Acknowledgments.} The authors would like to thank Andrei Agrachev and Rémi Carles for stimulating discussions.

This project has received financial support from the CNRS through the MITI interdisciplinary programs.

Karine Beauchard acknowledges support from grants ANR-20-CE40-0009 (Project TRECOS) and ANR-11-LABX-0020 (Labex Lebesgue), as well as from the Fondation Simone et Cino Del Duca -- Institut de France.

\bibliographystyle{siamplain}
\bibliography{references}

\end{document}